\RequirePackage{etex} 
\documentclass[12pt,final]{amsart}

\usepackage[utf8]{inputenc}
\usepackage{amssymb,amsthm,shuffle,mathdots,stmaryrd,mathtools,mathrsfs}
\usepackage{hyperref}
\usepackage{cleveref}
\usepackage[notref]{showkeys}
\usepackage[notref]{showkeys}
\usepackage{seqsplit}
\usepackage{xstring}
\renewcommand*\showkeyslabelformat[1]{%
\noexpandarg%
\StrSubstitute{#1}{ }{\textvisiblespace}[\TEMP]%
\fbox{\parbox[t]{\marginparwidth}{\raggedright\normalfont\scriptsize\ttfamily\expandafter\seqsplit\expandafter{\TEMP}}}}
\usepackage{autonum}
\makeatletter
\newcommand{\restore@Environment}[1]{%
  \AtBeginDocument{%
    \csletcs{#1*}{#1}%
    \csletcs{end#1*}{end#1}%
  }%
}
\forcsvlist\restore@Environment{alignat,equation,gather,multline,flalign,align}
\makeatother
\usepackage{bm,bbm}
\usepackage{ytableau}
\usepackage{geometry}
\usepackage{enumitem}
\theoremstyle{plain}
\newtheorem{thm}{Theorem}[section]
\newtheorem{prop}[thm]{Proposition}
\newtheorem{lemma}[thm]{Lemma}
\newtheorem{cor}[thm]{Corollary}

\theoremstyle{definition}
\newtheorem{defn}[thm]{Definition}
\newtheorem{remark}[thm]{Remark}
\newtheorem{example}[thm]{Example}

\crefname{thm}{Theorem}{Theorems}
\crefname{prop}{Proposition}{Propositions}
\crefname{lemma}{Lemma}{Lemmas}
\crefname{cor}{Corollary}{Corollaries}
\crefname{ques}{Question}{Questions}
\crefname{conj}{Conjecture}{Conjectures}

\crefname{defn}{Definition}{Definitions}
\crefname{remark}{Remark}{Remarks}
\crefname{example}{Example}{Examples}
\crefname{prob}{Problem}{Problems}

\crefname{enumi}{}{}
\crefname{enumii}{}{}
\crefname{enumiii}{}{}

\crefformat{equation}{{\upshape(#2#1#3)}}

\numberwithin{equation}{section}

\ytableausetup{aligntableaux=center}

\newcommand{\svdots}{\raisebox{-2pt}{$\vdots$}} 
\newcommand{\vsmall}{\rotatebox[origin=c]{-90}{$<$}}

\newcommand{\abs}[1]{\lvert #1 \rvert}

\newenvironment{varray}
{\biggl(\setlength{\arraycolsep}{1pt}\begin{matrix}}
{\end{matrix}\biggr)}

\allowdisplaybreaks

\usepackage{tikz}
\usetikzlibrary{calc}
\usetikzlibrary{decorations.pathreplacing}

\tikzset{vertex_black/.style={
circle, draw, black, fill=black,
inner sep=0pt, minimum width=4pt, label distance=1mm }}
\tikzset{vertex_white/.style={
circle, draw, black, fill=white,
inner sep=0pt, minimum width=4pt, label distance=1mm }}
\tikzset{vertex_dots/.style={yshift=0,inner sep=2,minimum width=0pt}}

\DeclareMathOperator{\dep}{dep}
\DeclareMathOperator{\wt}{wt}
\newcommand{\adm}{\mathrm{adm}}
\newcommand{\tot}{\mathrm{tot}}

\DeclareMathOperator{\SSYT}{\operatorname{SSYT}}
\DeclareMathOperator{\YT}{\operatorname{YT}}

\DeclareMathOperator{\Ker}{Ker}

\newcommand{\tsh}{\mathbin{\widetilde{\shuffle}}}
\newcommand{\SSD}{\mathrm{SSD}}

\newcommand{\Z}{\mathbb{Z}}
\newcommand{\Q}{\mathbb{Q}}
\newcommand{\R}{\mathbb{R}}

\newcommand{\bk}{{\boldsymbol{k}}}
\newcommand{\bb}{{\boldsymbol{b}}}
\newcommand{\bl}{{\boldsymbol{l}}}
\newcommand{\bmm}{{\boldsymbol{m}}}
\newcommand{\bn}{{\boldsymbol{n}}}
\newcommand{\bp}{{\boldsymbol{p}}}
\newcommand{\bw}{{\boldsymbol{w}}}

\newcommand{\sX}{\mathscr{X}}
\newcommand{\sT}{\mathscr{T}}

\title{Sum formulas for Schur multiple zeta values}

\author[H. Bachmann]{Henrik Bachmann}
\address{Graduate School of Mathematics,  Nagoya University, Nagoya, Japan.}
\email{henrik.bachmann@math.nagoya-u.ac.jp}

\author[S. Kadota]{Shin-ya Kadota}
\address{Faculty of Fundamental Science, National Institute of Technology (KOSEN), Niihama College, Niihama, Japan.}
\email{s.kadota@niihama-nct.ac.jp}

\author[Y. Suzuki]{Yuta Suzuki}
\address{Department of Mathematics, Rikkyo University, Tokyo, Japan}
\email{suzuyu@rikkyo.ac.jp}

\author[S. Yamamoto]{Shuji Yamamoto}
\address{Faculty of Science and Technology, Keio University, Yokohama, Japan}
\email{yamashu@math.keio.ac.jp}

\author[Y. Yamasaki]{Yoshinori Yamasaki}
\address{Graduate School of Science and Engineering, Ehime University, Matsuyama, Japan}
\email{yamasaki@math.sci.ehime-u.ac.jp}

\date{\today}

\begin{document}

\maketitle

\begin{abstract}
In this paper, we study sum formulas for Schur multiple zeta values and give a generalization of the sum formulas for multiple zeta(-star) values. We show that for ribbons of certain types, the sum of Schur multiple zeta values over all admissible Young tableaux of this shape evaluates to a rational multiple of the Riemann zeta value. For arbitrary ribbons with $n$ corners, we show that such a sum can be always expressed in terms of multiple zeta values of depth $\leq n$. In particular, when $n=2$, we give explicit, what we call, bounded type sum formulas for these ribbons. 
Finally, we show how to evaluate this sum when the corresponding Young diagram has exactly one corner and also prove bounded type sum formulas for them. This will also lead to relations among sums of Schur multiple zeta values over all admissible Young tableaux of different shapes. 
\end{abstract}

\section{Introduction}

The purpose of this note is to present several different types of sum formulas for Schur multiple zeta values, which can be seen as generalizations of classical sum formulas for multiple zeta(-star) values. Schur multiple zeta values are real numbers introduced in \cite{NakasujiPhuksuwanYamasaki2018}, and they can be seen as a simultaneous generalization of the multiple zeta values (MZVs) and multiple zeta-star values (MZSVs), which are defined for an index $\bk=(k_1,\dots,k_d) \in \Z_{\geq 1}^d$ with $k_d\geq 2$ by 
\begin{equation} \label{eq:mzv}
\zeta(\bk)\coloneqq\sum_{0<m_1<\cdots<m_d} \frac{1}{m_1^{k_1}\cdots m_d^{k_d}} \,, \qquad \zeta^\star(\bk)\coloneqq\sum_{0<m_1\leq\cdots \leq m_d} \frac{1}{m_1^{k_1}\cdots m_d^{k_d}} \,.
\end{equation}
Here the condition $k_d\geq 2$ ensures the convergence of the above sums, and the index $\bk$ is called admissible in this case. For an index $\bk = (k_1,\dots,k_d)$ we write $\wt(\bk)=k_1+\dots+k_d$ to denote its weight and $\dep(\bk)=d$ for its depth. A classical result is then (\cite{Granville1997},\cite{Hoffman92}), that  the sum of MZ(S)Vs over all admissible indices of fixed weight $w$ and depth $d$ evaluates to (an integer multiple of) $\zeta(w)$, i.e., for $d\geq 1$ and $w\geq d+1$
\begin{align} \label{eq:mzvsumformula}
    \sum_{\substack{\bk \text{ admissible}\\\wt(\bk)=w\\\dep(\bk)=d}} \zeta(\bk) = \zeta(w), \qquad\sum_{\substack{\bk \text{ admissible}\\\wt(\bk)=w\\\dep(\bk)=d}} \zeta^\star(\bk) = \binom{w-1}{d-1}\zeta(w)\,.
\end{align}
Schur MZVs generalize MZ(S)Vs by replacing an index $\bk$ by a skew Young tableau (see Definition \ref{def:schurmzv} for the exact definition). For example, if we have a skew Young diagram $\lambda\slash\mu = (2,2) \slash (1) = \ytableausetup{centertableaux, boxsize=0.6em}\begin{ytableau}
 \none & \, \\
  \, &  \,
\end{ytableau}$ and $k_1, k_3 \geq 1$, $k_2\geq 2$ the Schur MZV of \emph{shape} $\lambda\slash\mu$ for the Young tableau ${\footnotesize \ytableausetup{centertableaux, boxsize=1.3em} \bk = \begin{ytableau}
 \none & k_1 \\
  k_3 &  k_2
\end{ytableau} }\in \YT(\lambda\slash\mu)$ is defined by 
\ytableausetup{centertableaux, boxsize=1.3em}
\begin{align}
\zeta(\bk) = \zeta\left(\ {\footnotesize \begin{ytableau}
\none & k_1 \\
k_3 &  k_2
\end{ytableau}}\ \right) = \sum_{
	\arraycolsep=1.4pt\def\arraystretch{0.8}
	{\footnotesize \begin{array}{ccc}
	& &\,\,\,\,\,\,m_1 \\
    & &\,\,\,\,\,\vsmall  \\
	& m_3 &\leq m_2  
	\end{array} }} \frac{1}{m_1^{k_1}m_2^{k_2}m_3^{k_3}}\,,
\end{align}
where in the sum $m_1,m_2,m_3\geq 1$. 
Schur MZVs generalize MZVs and MZSVs in the sense that we recover them as the special cases
where the Young diagrams are respectively given by 
$(\underbrace{1,\ldots,1}_{d})$ and  $(d)$, i.e.,
\[ \zeta(k_1,\dots,k_d) = \zeta\left(\ {\footnotesize \ytableausetup{centertableaux, boxsize=1.3em}
	\begin{ytableau}
	k_1  \\
	\svdots \\
	k_d
	\end{ytableau}}\ \right) \quad \text{ and }\qquad \zeta^\star(k_1,\dots,k_d) =\zeta\left(\ {\footnotesize \ytableausetup{centertableaux, boxsize=1.3em}
	\begin{ytableau}
	k_1 & \cdots &k_d
	\end{ytableau}}\ \right) \,.   \]
The classical sum formulas \eqref{eq:mzvsumformula} state that for the shapes $\lambda=(1,\dots,1)$ or $\lambda=(d)$ the sum of Schur MZVs over all admissible Young tableaux of shape $\lambda$ and weight $w\geq d+1$ evaluates to (a multiple of) $\zeta(w)$. Therefore, it is natural to ask, how this situation is for a general shape $\lambda \slash \mu$, i.e., if there exist explicit evaluations of 
\begin{align}\label{eq:sk}
    S_w(\lambda \slash \mu) \coloneqq \sum_{\substack{\bk \in \YT(\lambda \slash \mu) \\ \bk \text{ admissible}\\\wt(\bk)=w}} \zeta(\bk)\,.
\end{align}
The optimistic guess that $ S_w(\lambda \slash \mu)$ is always a rational multiple of $\zeta(w)$ seems to be wrong since we will see that, for example, for $\lambda = (2,2)=\ytableausetup{centertableaux, boxsize=0.4em}
	\begin{ytableau}
	\, & 	\,  \\
	\, & 	\,
	\end{ytableau}$ and $\mu =\varnothing$ we have for $w\geq 5$
\begin{equation}\label{eq:22square}\begin{split}
   S_w\left(\ {\footnotesize \ytableausetup{centertableaux, boxsize=0.8em}
	\begin{ytableau}
	\, & 	\,  \\
	\, & 	\,
	\end{ytableau}} \ \right)
  &= -(w-2) \zeta(1,w-1) + (w-4) \zeta(2,w-2)+2\zeta(3,w-3)\\
  & \quad -2\zeta(3)\zeta(w-3)+(w-2)\zeta(2)\zeta(w-2)\,,
	\end{split}
\end{equation}
which, by computer experiments, seems not to evaluate to a rational multiple of $\zeta(w)$ for arbitrary $w$. Nevertheless, we see that \eqref{eq:22square} gives a representation as a sum of products of MZVs, where the number of terms does not depend on $w$. That such an expression exists is not obvious since a priori, the number of terms in \eqref{eq:sk} increases with the weight $w$. It is, therefore, also reasonable to call \eqref{eq:22square} a sum formula. Since it is of a different type to the sum formula of MZVs, we introduce the following 
types of sum formulas for shape $\lambda \slash \mu$: 
\begin{enumerate}[leftmargin=4cm]
    \item[\emph{single type}:] $S_w(\lambda \slash \mu)$ evaluates to a rational multiple of $\zeta(w)$.
    \item[\emph{polynomial type}:] $S_w(\lambda \slash \mu)$ evaluates to a polynomial in $\zeta(m)$ with $m\leq w$.
    \item[\emph{bounded type}:] $S_w(\lambda \slash \mu)$ is expressed as a $\Q$-linear combination of MZVs, where the number of terms does not depend\footnote{By this, we mean that for a fixed $\lambda/\mu$ we can find $a( \bk ,w) \in \mathbb{Q}$ for admissible indices $\bk$ and $w\geq 1$ with $S_{w}(\lambda/\mu)=\sum_{ \bk } a(\bk,w) \zeta(\bk)$ such that there exists a positive real number $C$ with $\abs{\{ \bk \mid
a(\bk,w)\neq 0\}} < C$ for all $w\geq 1$.} on $w$, but just on $\lambda \slash \mu$.
\end{enumerate}
With this terminology, the classical sum formulas are of single type, and the formula \eqref{eq:22square} is of bounded type. 
This list is not intended to be exhaustive, i.e., we do not claim that any shape $\lambda/\mu$ has a sum formula of one of these three types. 
It is also not exclusive, for instance, a single type formula is also both of polynomial and bounded type. 

In this note, we will focus on the case when $\lambda \slash \mu$ has either just one corner (see Section~\ref{subsec:notation} for the definition of corners) or $\lambda \slash \mu$ is a ribbon, by which we mean a (skew) Young diagram which is connected and contains no $2\times 2$ block of boxes. For example, if $\lambda$ is just a column or a row, we obtain the single type sum formulas \eqref{eq:mzvsumformula}. As a generalization we will show (Theorem \ref{thm:anti-hook}) that for $\lambda \slash \mu = (\underbrace{d, \dots, d}_{r}) \slash (\underbrace{d-1, \dots, d-1}_{r-1})$, i.e., when $\lambda \slash \mu$ is an anti-hook, we have
\begin{align}
     S_w(\lambda \slash \mu) =   S_w\left(\ {\footnotesize \ytableausetup{centertableaux, boxsize=1.2em}
	\begin{ytableau}
\none & \none  & \, \\
\none & \none  & \svdots \\
\, & \cdots  & \,
	\end{ytableau}}\ \right) &= \binom{w-1}{d-1}\zeta(w)\,.
\end{align}
This can be seen as a unification of the two classical sum formulas \eqref{eq:mzvsumformula}. In Theorem \ref{thm:stair 1}, we will show that also the ``stairs of tread one'', where the height of each stair is the same,  will give rise to single type sum formulas. For example, as special cases, we get the following formulas, valid for all $w\geq 8$
\begin{align}\label{eq:exampletreadone}
 S_w\left(\ \ytableausetup{centertableaux, boxsize=0.8em}
\begin{ytableau}
\none & \none  & \, \\
\none & \none  & \, \\
\none & \,  & \, \\
\none & \,  & \none \\
\, & \,  & \none \\
\end{ytableau} \ \right)= \frac{(w-8)(w-1)}{2}\zeta(w) \,,\quad 
 S_w\left(\, \ytableausetup{centertableaux, boxsize=0.8em}
\begin{ytableau}
\none & \none  &\none  & \, \\
\none & \none  &\,  & \, \\
\none & \,  &\,  & \none \\
\, & \,  &\none  & \none \\
\end{ytableau}\, \right)= \frac{(w-9)(w-8)(w-1)}{6}\zeta(w) \,.
\end{align}
For arbitrary ribbons with $n$ corners, we will see in Corollary \ref{cor:ncornerribbon} that the corresponding $S_w$ can be always expressed in terms of MZVs of depth $\leq n$.
 Especially, in Theorem \ref{thm:ribbon_2_corner} we give an explicit bounded type sum formula 
 for any ribbon with two corners in terms of double zeta values.

In the last section, we consider shapes that have exactly one corner but which are not necessarily ribbons. We will show that arbitrary shapes with one corner always give rise to sum formulas of bounded type  (Theorem \ref{thm:onecornerwithphi}) similar to \eqref{eq:22square}. This will also lead to relations among $S_w(\lambda / \mu)$ for different shapes $\lambda/ \mu$. For example, as a special case of Theorem \ref{thm:S_w rel}, we will see that for any $w\geq 1$
\begin{align*}
\ytableausetup{centertableaux, boxsize=0.8em}
  2  S_w\!\!\left(\, \begin{ytableau}
	\none & 	\none  & 	\,  \\
    \none & 	\none  & 	\, \\
    \none & 	\,  & 	\, \\
	*(lightgray)\, & 	\, & 	\,
\end{ytableau}\,\right)
  -  S_w\!\!\left(\, \begin{ytableau}
	\none & 	\,  \\
   *(lightgray) \, & 	\,  \\
    \, & 	\,  \\
	\, & 	\,
	\end{ytableau}\,\right)
     -4 S_w\!\!\left(\, \begin{ytableau}
	\none & 	*(lightgray)\,  \\
    \none & 	\,  \\
    \none & 	\,  \\
    \, & 	\,  \\
	\, & 	\,
	\end{ytableau}\,\right)
 = (w-7) S_w\!\!\left(\,  \begin{ytableau}
	\none & 	\,  \\
    \none & 	\,  \\
    \, & 	\,  \\
	\, & 	\,
	\end{ytableau}\,\right)\,.
\end{align*}
Notice that the diagrams appearing in the left-hand side are obtained by adding one box (the gray ones) to the diagram in the right-hand side. 

The structure of this paper is as follows: In Section \ref{sec:weightedsum} we give a recursive formula for a certain weighted sums of MZVs, which will be used in later sections. For this, we will recall the notion of $2$-posets and their associated integrals. In Section \ref{sec:ribbons}, we consider $S_w(\lambda / \mu)$ where $\lambda / \mu$ are ribbons. Section \ref{sec:OneCorner} will be devoted to the case of shapes with exactly one corner.

\subsection{Notation and definition of Schur MZVs}\label{subsec:notation}
We will use the following notation in this work. A tuple $\bk=(k_1,\dots,k_d)\in \Z^d_{\geq 1}$ will be called an \emph{index} and we write $\bk=\varnothing$ when $d=0$. 
We call $\bk$ \emph{admissible} if $k_d\geq 2$ or if $\bk=\varnothing$. For $n\geq 1$ and $k\in \Z_{\ge 1}$ we use for the $n$-time repetition of $k$ the usual notation  $\{k\}^n = \underbrace{k,\dots,k}_{n}$.

A \emph{partition} of a natural number $n$ is a tuple $\lambda = (\lambda_1,\dots,\lambda_h)$ of positive integers
$\lambda_1 \geq \dots \geq \lambda_h \geq 1$ with $n = |\lambda|= \lambda_1 + \dots + \lambda_h$. We will also use the notation $\lambda=(n^{m_n},\ldots,2^{m_2},1^{m_1})$, where $m_i=m_i(\lambda)$ is the multiplicity of $i$ in $\lambda$.
For another partition $\mu=(\mu_1,\dots,\mu_r)$ we write $\mu \subset \lambda$ if $r\leq h$ and $\mu_i\le\lambda_i$ for $i=1,\dots,r$, and we define the \emph{skew Young diagram} $D(\lambda/\mu)$ of $\lambda \slash \mu$ by 
\[D(\lambda \slash \mu) = \left\{(i,j) \in \Z^2 \mid 1 \leq i \leq h\,, \mu_i < j \leq \lambda_i \right\},\]
where $\mu_i = 0$ for $i>r$. In the case where $\mu=\varnothing$ is the empty partition (i.e., the unique partition of zero) we just write $\lambda\slash\mu = \lambda$. 

A \emph{Young tableau} $\bk =  (k_{i,j})_{(i,j) \in D(\lambda \slash \mu)}$ of shape $\lambda \slash \mu$ is a filling of $D(\lambda\slash\mu)$ obtained by putting $k_{i,j}\in\Z_{\geq 1}$ into the $(i,j)$-entry of $D(\lambda\slash\mu)$. For shorter notation, we will also just write $(k_{i,j})$ in the following if the shape $\lambda \slash \mu$ is clear from the context. 
A Young tableau $(m_{i,j})$ is called \emph{semi-standard} if $m_{i,j}<m_{i+1,j}$ and $m_{i,j}\leq m_{i,j+1}$ for all possible $i$ and $j$.
The set of all Young tableaux and all semi-standard Young tableaux of shape $\lambda \slash \mu$ are denoted by $\YT(\lambda \slash \mu)$ and 
$\SSYT(\lambda \slash \mu)$, respectively. 

An entry $(i,j)\in D(\lambda\slash\mu)$ is called an (outer) \emph{corner} of $\lambda \slash \mu$
if $ (i,j+1) \not\in D(\lambda\slash\mu)$ and $ (i+1,j) \not\in D(\lambda\slash\mu)$.
We denote the set of all corners of $\lambda \slash \mu$ by $C(\lambda \slash \mu)$. 
For a Young tableau $\bk = (k_{i,j}) \in\YT(\lambda \slash \mu)$ we define its \emph{weight} by $\wt(\bk) = \sum_{(i,j) \in D(\lambda \slash \mu)} k_{i,j}$
and we call it \emph{admissible} if $k_{i,j} \geq 2$ for all $(i,j) \in C(\lambda\slash\mu)$. 

\begin{defn}\label{def:schurmzv}
For an admissible $\bk = (k_{i,j}) \in \YT(\lambda \slash \mu)$
the \emph{Schur multiple zeta value} (Schur MZV) is defined by  
\begin{align}\label{eq:defschurmzv}
    \zeta(\bk) \coloneqq \sum_{(m_{i,j}) \in \SSYT(\lambda \slash \mu)} \prod_{(i,j) \in D(\lambda \slash \mu)}  \frac{1}{m_{i,j}^{k_{i,j}}} \,.
\end{align}
\end{defn}
Note that the admissibility of $\bk$ ensures the convergence of \eqref{eq:defschurmzv} (\cite[Lemma 2.1]{NakasujiPhuksuwanYamasaki2018}). For the empty tableau $\bk=\varnothing$, we have $\zeta(\varnothing)=1$. 

Finally, we mention that the convention for the binomial coefficients we use in this work is for $n,k \in \Z$ given by 
\begin{align*}
    \binom{n}{k} \coloneqq\begin{cases} \frac{n (n-1) \cdots (n-(k-1))}{k!}& \text{if }k>0,\\
    1& \text{if }k=0,\\
    0& \text{if }k<0.
    \end{cases}
\end{align*}

\subsection*{Acknowledgement}  
The authors thank the referee for his/her careful reading and valuable comments which improved the quality of the manuscript.
The first author was partially supported by JSPS KAKENHI Grant Numbers JP19K14499, JP21K13771.
The third author was partially supported by JSPS KAKENHI Grant Numbers JP19K23402, JP21K13772.
The fourth author was partially supported by JSPS KAKENHI Grant Numbers JP18H05233, JP18K03221, JP21K03185.
The fifth author was partially supported by JSPS KAKENHI Grant Numbers JP21K03206.


\section{Weighted sum formulas} 
\label{sec:weightedsum}
When evaluating sums of Schur MZVs we will often encounter weighted sums of MZVs, which we will discuss in this section. For indices $\bn=(n_1,\ldots,n_d)$, $\bk=(k_1,\ldots,k_d)$ and an integer $l\ge 0$, define
\begin{align}
P_l(\bn;\bk)
&\coloneqq \sum_{\substack{\bw=(w_1,\ldots,w_d): \text{ adm.}\\w_i\ge n_i \ (i=1, \dots, d)\\
\wt(\bw)=\wt(\bk)+l}
}
\prod^{d}_{i=1}\binom{w_i-n_i}{k_i-1}\cdot \zeta(w_1,\ldots,w_d).
\end{align}
Notice that by definition 
$P_l(\varnothing;\varnothing)=1$ if $l=0$ and $0$ otherwise.  
In particular, we put  
$P_l(\bk) \coloneqq P_l((1,\ldots,1);\bk)$
and $Q_l(\bk) \coloneqq P_l((1,\ldots,1,2);\bk)$.
The aim of this section is to obtain explicit bounded expressions of $P_l(\bk)$ and $Q_l(\bk)$, which play important roles throughout the present paper. Here, we say that an expression of $P_l(\bk)$ or $Q_l(\bk)$ is {\it bounded} if the number of terms appearing in the expression does not depend on $l$ but only on $\bk$. Notice that, since $\binom{w-2}{k-1}=\sum^{k-1}_{j=0}(-1)^j\binom{w-1}{k-j-1}$, we have 
\begin{equation}
\label{for:PtoQ}
Q_l(\bk)
=\sum^{k_d-1}_{j=0}(-1)^jP_{l+j}(k_1,\ldots,k_{d-1},k_d-j)\,
\end{equation}
and in particular $Q_l(\bk)=P_l(\bk)$ if $\bk$ is non-admissible. 
By \eqref{for:PtoQ}, it is sufficient to study only $P_l(\bk)$. 
Moreover, we may assume that $l>0$ because the case $l=0$ is trivial: $P_0(\bk)=\zeta(\bk)$ if $\bk$ is admissible and $0$ otherwise, and $Q_0(\bk)=0$ for any index $\bk$. Furthermore, when $d=1$, we have for $k\ge 1$ and $l>0$
\[
P_l(k)
=\binom{l+k-1}{k-1}\zeta(l+k)\,,
\quad 
Q_l(k)
=\binom{l+k-2}{k-1}\zeta(l+k)
\]
and therefore also assume that $d\ge 2$.

The next proposition asserts that $P_l(\bk)$ satisfies recursive formulas with respect to $l$, which can be described by using the Schur MZVs of anti-hook shape 
\begin{align}\begin{split}\label{eq:antihookzeta}
\zeta\begin{varray}
\bl \\ \bk 
\end{varray}
={}&
\zeta
\begin{varray}
l_1,\ldots,l_s\\
k_1,\ldots,k_r
\end{varray}\\
\coloneqq{}&
\ytableausetup{boxsize=1.2em}
\zeta\left(\;\begin{ytableau}
\none & \none & \none & k_1 \\
\none & \none & \none & \svdots \\
l_1 & \cdots & l_s & k_r
\end{ytableau}\;\right)
=
\sum_{0<a_1<\cdots<a_r\ge b_s\ge\cdots\ge b_1>0}
\frac{1}{a_1^{k_1}\cdots a_r^{k_r}b_1^{l_1}\cdots b_s^{l_s}}
\end{split}
\end{align}
with $\bl=(l_1,\ldots,l_s)$ being an index and $\bk=(k_1,\ldots,k_r)$ a non-empty admissible index. 

\begin{prop}
\label{prop:bwsumformula}
Let $d\ge 2$ and $l>0$.
\begin{enumerate}[label=\textup{(\roman*)}]
\item 
If $\bk=(k_1,\ldots,k_d)$ is admissible, then it holds that  
\begin{equation}
\label{eq:general weighted sum adm}
P_{l}(\bk)
=\sum_{i=1}^d\sum_{a_i=0}^{k_i-1} (-1)^{k_1+\cdots+k_{i-1}+a_i}
P_{k_i-1-a_i}(k_{i-1},\ldots,k_1,l+1)\, P_{a_i}(k_{i+1},\ldots,k_d)\,. 
\end{equation}
\item 
If $\bk=(k_1,\ldots,k_d)$ is non-admissible (i.e., $k_d=1$), then it holds that
\begin{equation}
\label{eq:general weighted sum nonadm}
\begin{aligned}
&P_{l}(\bk)
=\sum^{d}_{i=1}\sum_{a_{i}=0}^{k_{i}-1}
(-1)^{k_1+\cdots+k_{i-1}+a_{i}}P_{k_{i}-1-a_i}(k_{i-1},\ldots,k_{1},l+1)P_{a_i}(k_{i+1},\ldots,k_{d-1},1)\\
&\ \ +\sum_{i=1}^{d-1}(-1)^{l+d+k_i}
\sum_{\substack{(b_0,\ldots,b_{d-1})\in\Z_{\ge 1}^d\\
b_i=2\\
b_0+\cdots+b_{d-1}=\wt(\bk)+l+1}}(-1)^{b_0+b_1+\cdots+b_{i-1}}
\binom{b_0-1}{l}
\Biggl\{\prod^{d-1}_{\substack{j=1 \\ j\ne i}}\binom{b_j-1}{k_j-1}\Biggr\}\\
&\ \ \times 
\sum_{j=i}^{d-1}\sum_{c_j=1}^{b_j-1}(-1)^{c_j+j+b_{j+1}+\cdots +b_{d-1}}
\zeta\begin{varray}
c_j,b_{j+1},\ldots,b_{d-1}\\
b_{i-1},\ldots,b_1,b_0
\end{varray} \zeta(b_{i+1},\ldots,b_{j-1},b_{j}-c_{j}+1)\,.
\end{aligned}
\end{equation}
\end{enumerate}
\end{prop}

\begin{remark}\label{rem:plisbounded}
The number of terms in the expression 
\eqref{eq:general weighted sum nonadm}
is actually bounded 
by a constant depending on $\bk$ but not on $l$ 
by the following reason: By the non-vanishing of the binomial coefficient $\binom{b_0-1}{l}$,
the summation variable $b_0$ can be restricted by $b_0\ge l+1$
and then the other summation variables are bounded independently of $l$.
Also, by applying the formula \eqref{eq:zeta(bk) via SSD} to an anti-hook, we can expand each term 
\[\zeta\begin{varray}
c_j,b_{j+1},\ldots,b_{d-1}\\
b_{i-1},\ldots,b_1,b_0
\end{varray}\]
into a sum of MZVs. The number of appearing MZVs is independent of $b_0,\ldots,b_{d-1}$ and $c_j$. 
Finally, by using the harmonic product formula,
we can rewrite the products of MZVs into a sum of MZVs.
The resulting number of terms, after the application of the harmonic product formula, depends only on the 
number of entries of MZVs appearing as above and so it is independent of $l$.
\end{remark}

To prove Proposition \ref{prop:bwsumformula}, we first recall the notion of 2-posets and the associated integrals introduced by the fourth-named author in \cite{YamamotoIntegral}.

\begin{defn}[{\cite[Definition~2.1]{YamamotoIntegral}}]\leavevmode
\begin{enumerate}[label=\textup{(\roman*)}]
\item A $2$-poset is a pair $(X,\delta_X)$, 
where $X=(X,\le)$ is a finite partially ordered set (poset for short) 
and $\delta_X$ is a map (called the label map of $X$) from $X$ to $\{0,1\}$. 
We often omit $\delta_X$ and simply say ``a $2$-poset $X$''. 
Moreover, a 2-poset $X$ is called admissible if $\delta_X(x)=0$ for all maximal elements $x$ 
and $\delta_X(x)=1$ for all minimal elements $x$.
\item For an admissible $2$-poset $X$, the associated integral $I(X)$ is defined by
\begin{equation}
\label{def:Yintegral}
I(X)=\int_{\Delta_X}\prod_{x\in X}\omega_{\delta_X(x)}(t_x)\,,
\end{equation}
where $\Delta_X=\left\{\left.(t_x)_{x}\in [0,1]^{X}\,\right|\,\text{$t_x<t_y$ if $x<y$}\right\}$ and 
$\omega_0(t)=\frac{dt}{t}$ and $\omega_1(t)=\frac{dt}{1-t}$.
\end{enumerate}
\end{defn}

We depict a $2$-poset as a Hasse diagram in which an element $x$ with $\delta_X(x)=0$ (resp. $\delta_X(x)=1$) is represented by $\circ$ (resp. $\bullet$). For example, the diagram 
\begin{align}
\label{ex:YI}
\begin{tikzpicture}[thick,x=10pt,y=10pt,baseline=(base)]
\coordinate (base) at (3,2);
\node[vertex_black] (B1) at (0,0) {};
\node[vertex_white] (W11) at (0,1) {}; 
\node[vertex_black] (B2) at (0,2) {};
\node[vertex_black] (B3) at (0,3) {};
\node[vertex_white] (W31) at (0,4) {};
\node[vertex_black] (B4) at (2,2) {};
\node[vertex_white] (W41) at (3,3) {};
\node[vertex_white] (W42) at (4,4) {};
\node[vertex_black] (B5) at (5,3) {};
\node[vertex_white] (W51) at (6,4) {};
\draw{
(B1)--(W11)--(B2)--(B3)--(W31)--(B4)--(W41)--(W42)--(B5)--(W51)
};
\end{tikzpicture}
\end{align}
represents the $2$-poset $X=\{x_1,\ldots,x_{10}\}$ with order $x_1<x_2<x_3<x_4<x_5>x_6<x_7<x_8>x_9<x_{10}$ and label $(\delta_X(x_1),\ldots,\delta_X(x_{10}))=(1,0,1,1,0,1,0,0,1,0)$.

In \cite{KanekoYamamoto2018}, it is shown that the Schur MZVs of anti-hook shape has the following expression by the associated integral of a $2$-poset. This can be regarded as simultaneous generalization of the integral expressions of MZVs and MZSVs.

\begin{thm}[{\cite[Theorem~4.1]{KanekoYamamoto2018}}]
\label{thm:integral series identity}
For an index $\bl=(l_1,\ldots,l_s)$ and a non-empty admissible index $\bk=(k_1,\ldots,k_r)$, we have 
\begin{align}
\zeta\begin{varray}
\bl \\ \bk 
\end{varray}
=
I
\left(
\begin{tikzpicture}[thick,x=8pt,y=8pt,baseline=44pt]
\node[vertex_black] (Bi) at (0,0) {};
\node[vertex_white] (Wi-11) at (0,1) {};
\node[vertex_white] (Wi-12) at (0,3) {};
\node[vertex_black] (B2) at (0,5) {}; 
\node[vertex_white] (W11) at (0,6) {};
\node[vertex_white] (W12) at (0,8) {};
\node[vertex_black] (B1) at (0,9) {};
\node[vertex_white] (W01) at (0,10) {}; 
\node[vertex_white] (W02) at (0,12) {};
\node[vertex_black] (Bd) at (3,9) {};
\node[vertex_white] (Wd1) at (4,10) {}; 
\node[vertex_white] (Wd2) at (6,12) {}; 
\node[vertex_black] (Bi+2) at (10,9) {};
\node[vertex_white] (Wi+21) at (11,10) {}; 
\node[vertex_white] (Wi+22) at (13,12) {}; 
\draw{
(Bi)--(Wi-11) (B2)--(W11) (W12)--(B1)--(W01) (W02)--(Bd)--(Wd1) 
(Wd2)--(7,10.5) (9,10.5)--(Bi+2)--(Wi+21) (Wi+22)
};
\draw[densely dotted] {
(Wi-11)--(Wi-12)--(B2) (W11)--(W12) (W01)--(W02) (Wd1)--(Wd2) 
(7.3,10.5)--(8.7,10.5) (Wi+21)--(Wi+22) 
};
\draw[decorate,decoration={brace,amplitude=3},xshift=-2pt,yshift=0pt]
(0,1)--(0,3) node [midway,xshift=-18pt,yshift=0pt]{\scriptsize $k_{1}-1$};
\draw[decorate,decoration={brace,amplitude=3},xshift=-2pt,yshift=0pt]
(0,6)--(0,8) node [midway,xshift=-24pt,yshift=0pt]{\scriptsize $k_{r-1}-1$};
\draw[decorate,decoration={brace,amplitude=3},xshift=-2pt,yshift=0pt]
(0,10)--(0,12) node [midway,xshift=-18pt,yshift=0pt]{\scriptsize $k_r-1$};
\draw[decorate,decoration={brace,amplitude=3},xshift=-2pt,yshift=2pt]
(4,10)--(6,12) node [midway,xshift=-15pt,yshift=4pt]{\scriptsize $l_{s}-1$};
\draw[decorate,decoration={brace,amplitude=3},xshift=-2pt,yshift=2pt]
(11,10)--(13,12) node [midway,xshift=-15pt,yshift=4pt]{\scriptsize $l_{1}-1$};
\end{tikzpicture}
\ \right)\,. 
\end{align}
\end{thm}

For example, for the $2$-poset $X$ given by \eqref{ex:YI}, we have 
$\zeta\begin{varray}
2,3 \\ 2,1,2 
\end{varray}
=I(X)$.

In our proof of \cref{prop:bwsumformula}, we consider a kind of extention of the integral $I(X)$ 
to non-admissible $2$-posets $X$. This extension is given by using the notion of ``admissible part'' 
which we define below. 

Let $\sX$ be the set of isomorphism classes of $2$-posets, 
and $\Q\sX$ denote the $\Q$-vector space freely generated by this set. 
We equip $\Q\sX$ with a $\Q$-algebra structure by setting $[X]\cdot[Y]\coloneqq[X\sqcup Y]$. 
If we let $\sX^0\subset\sX$ be the subset consisting of admissible $2$-posets, 
its $\Q$-span $\Q\sX^0$ becomes a $\Q$-subalgebra of $\Q\sX$ and 
the integral \eqref{def:Yintegral} defines a $\Q$-algebra homomorphism $I\colon\Q\sX^0\to\R$. 

Let $\sT\subset\sX$ be the subset of totally ordered $2$-posets. 
Then a $\Q$-linear map $\Q\sX\to\Q\sT$, which we call the \emph{totally ordered expansion}, 
is defined by 
\[[X]=[X,\le,\delta]\longmapsto [X]^\tot\coloneqq\sum_{\le'}[X,\le',\delta], \]
where $[X]=[X,\le,\delta]$ is the isomorphism class of any $2$-poset $X$ 
and $\le'$ runs over the total orders on the set $X$ 
which are refinements of the original partial order $\le$. 
We have $[X]^\tot=[X_a^b]^\tot+[X_b^a]^\tot$ for any $2$-poset $X$ and non-comparable elements $a,b\in X$, 
where $X^b_a$ denotes the 2-poset obtained from $X$ by adjoining the relation $a<b$. 
Note also that the integration map $I\colon\Q\sX^0\to\R$ factors through the totally ordered expansion, 
i.e., we have $I([X])=I([X]^\tot)$ for any $[X]\in\sX^0$. 

For any $2$-poset $X$, we define its \emph{admissible part} $[X]^\adm$ to be 
the partial sum of the totally ordered expansion $[X]^\tot$ consisting of the admissible terms. 
For example, if $X=
\begin{tikzpicture}[thick,x=10pt,y=10pt,baseline=10pt]
\node[vertex_white] (W1) at (-1,1) {}; 
\node[vertex_black] (B1) at (0,0) {};
\node[vertex_white] (W2) at (1,1) {}; 
\node[vertex_black] (B2) at (1,2) {};
\draw{
(W1)--(B1)--(W2)--(B2)
};
\end{tikzpicture}
\,$, we have 
\[
[X]^\tot=
\left[\ 
\begin{tikzpicture}[thick,x=10pt,y=10pt,baseline=10pt]
\node[vertex_black] (B1) at (0,0) {};
\node[vertex_white] (W1) at (0,1) {}; 
\node[vertex_black] (B2) at (0,2) {};
\node[vertex_white] (W2) at (0,3) {}; 
\draw{
(B1)--(W1)--(B2)--(W2)
};
\end{tikzpicture}
\ \right] 
+2
\left[\ 
\begin{tikzpicture}[thick,x=10pt,y=10pt,baseline=10pt]
\node[vertex_black] (B1) at (0,0) {};
\node[vertex_white] (W1) at (0,1) {}; 
\node[vertex_white] (W2) at (0,2) {}; 
\node[vertex_black] (B2) at (0,3) {};
\draw{
(B1)--(W1)--(W2)--(B2)
};
\end{tikzpicture}
\ \right]\, \text{ and }\  
[X]^{\adm}=
\left[\ 
\begin{tikzpicture}[thick,x=10pt,y=10pt,baseline=10pt]
\node[vertex_black] (B1) at (0,0) {};
\node[vertex_white] (W1) at (0,1) {}; 
\node[vertex_black] (B2) at (0,2) {};
\node[vertex_white] (W2) at (0,3) {}; 
\draw{
(B1)--(W1)--(B2)--(W2)
};
\end{tikzpicture}
\ \right]\,. \]
Then the $\Q$-linear map 
\[\Q\sX\longrightarrow\R;\ [X]\longmapsto I([X]^\adm)\]
is an extension of $I\colon\Q\sX^0\to\R$. 
Notice that this map is \emph{not} an algebra homomorphism (for example, $[\circ]\cdot [\bullet]\longmapsto
I\left(\left[\ 
\begin{tikzpicture}[thick,x=10pt,y=10pt,baseline=1.5pt]
\node[vertex_black] (B1) at (0,0) {};
\node[vertex_white] (W1) at (0,1) {}; 
\draw{
(B1)--(W1)
};
\end{tikzpicture}
\ \right]\right)$ but $[\circ],[\bullet]\longmapsto 0$), 
unlike the map defined by using the shuffle regularization (see Remark at the end of \S 2 of \cite{YamamotoIntegral}). 

In the following computations, we omit the symbol `$\tot$'; 
for example, we write $[X]=[X_a^b]+[X_b^a]$ instead of $[X]^\tot=[X_a^b]^\tot+[X_b^a]^\tot$. 
In other words, we compute in the quotient space $\Q\sX/\Ker([X]\mapsto[X]^\tot)$. 



\begin{proof}
[Proof of Proposition~\ref{prop:bwsumformula}]
Set $m_i=k_i-1$ for $1\le i\le d$. 
Define
\[
X=X_{l}(\bk)=
\begin{tikzpicture}[thick,x=8pt,y=8pt,baseline=44pt]
\node[vertex_black] (B1) at (0,0) {};
\node[vertex_white] (W01) at (-2,1) {}; 
\node[vertex_white] (W02) at (-2,4) {};
\node[vertex_white] (W11) at (2,1) {};
\node[vertex_white] (W12) at (2,3) {};
\node[vertex_black] (B2) at (2,4) {};
\node[vertex_white] (W21) at (2,5) {};
\node[vertex_white] (W22) at (2,7) {};
\node[vertex_black] (Bd) at (2,9) {};
\node[vertex_white] (Wd1) at (2,10) {};
\node[vertex_white] (Wd2) at (2,12) {};
\draw{
(B1)--(W01) (B1)--(W11) (W12)--(B2)--(W21) (Bd)--(Wd1)
};
\draw[densely dotted] {
(W01)--(W02) (W11)--(W12) (W21)--(W22)--(Bd) (Wd1)--(Wd2)
};
\draw[decorate,decoration={brace,amplitude=5},xshift=-2pt,yshift=0pt]
(-2,1)--(-2,4) node [midway,xshift=-8pt,yshift=0pt]{\scriptsize $l$};
\draw[decorate,decoration={brace,amplitude=3},xshift=2pt,yshift=0pt]
(2,3)--(2,1) node [midway,xshift=12pt,yshift=0pt]{\scriptsize $m_1$};
\draw[decorate,decoration={brace,amplitude=3},xshift=2pt,yshift=0pt]
(2,7)--(2,5) node [midway,xshift=12pt,yshift=0pt]{\scriptsize $m_2$};
\draw[decorate,decoration={brace,amplitude=3},xshift=2pt,yshift=0pt]
(2,12)--(2,10) node [midway,xshift=12pt,yshift=0pt]{\scriptsize $m_d$};
\end{tikzpicture}
\,.
\]
Then we have 
\begin{equation}
\label{for:P is I of X}
P_l(\bk)=I([X]^{\adm})\,,
\end{equation} 
which is the key ingredient of the proof.
On the other hand, we see that 
\begin{align}
[X]
&=\sum_{a=0}^{m_1}(-1)^a
\left[
\begin{tikzpicture}[thick,x=8pt,y=8pt,baseline=30pt]
\node[vertex_black] (B1) at (0,0) {};
\node[vertex_white] (W01) at (-2,1) {}; 
\node[vertex_white] (W02) at (-2,4) {};
\node[vertex_white] (W11) at (2,1) {};
\node[vertex_white] (W12) at (2,3) {};
\node[vertex_black] (B2) at (11,0) {};
\node[vertex_white] (W1) at (9,1) {};
\node[vertex_white] (W2) at (9,3) {};
\node[vertex_white] (W21) at (13,1) {};
\node[vertex_white] (W22) at (13,3) {};
\node[vertex_black] (Bd) at (13,5) {};
\node[vertex_white] (Wd1) at (13,6) {};
\node[vertex_white] (Wd2) at (13,8) {};
\draw{
(B1)--(W01) (B1)--(W11) (W1)--(B2)--(W21) (Bd)--(Wd1)
};
\draw[densely dotted] {
(W01)--(W02) (W11)--(W12) (W1)--(W2) (W21)--(W22)--(Bd) (Wd1)--(Wd2)
};
\draw[decorate,decoration={brace,amplitude=5},xshift=-2pt,yshift=0pt]
(-2,1)--(-2,4) node [midway,xshift=-8pt,yshift=0pt]{\scriptsize $l$};
\draw[decorate,decoration={brace,amplitude=3},xshift=2pt,yshift=0pt]
(2,3)--(2,1) node [midway,xshift=18pt,yshift=0pt]{\scriptsize $m_1\!-\!a$};
\draw[decorate,decoration={brace,amplitude=5},xshift=-2pt,yshift=0pt]
(9,1)--(9,3) node [midway,xshift=-10pt,yshift=0pt]{\scriptsize $a$};
\draw[decorate,decoration={brace,amplitude=3},xshift=2pt,yshift=0pt]
(13,3)--(13,1) node [midway,xshift=12pt,yshift=0pt]{\scriptsize $m_2$};
\draw[decorate,decoration={brace,amplitude=3},xshift=2pt,yshift=0pt]
(13,8)--(13,6) node [midway,xshift=12pt,yshift=0pt]{\scriptsize $m_d$};
\end{tikzpicture}
\ \right]
+(-1)^{k_1}
\left[
\begin{tikzpicture}[thick,x=8pt,y=8pt,baseline=30pt]
\node[vertex_black] (B2) at (0,0) {};
\node[vertex_white] (W11) at (-2,1) {}; 
\node[vertex_white] (W12) at (-2,3) {};
\node[vertex_black] (B1) at (-2,4) {};
\node[vertex_white] (W01) at (-2,5) {};
\node[vertex_white] (W02) at (-2,8) {};
\node[vertex_white] (W21) at (2,1) {};
\node[vertex_white] (W22) at (2,3) {};
\node[vertex_black] (Bd) at (2,5) {};
\node[vertex_white] (Wd1) at (2,6) {};
\node[vertex_white] (Wd2) at (2,8) {};
\draw{
(B2)--(W11) (W12)--(B1)--(W01) (B2)--(W21) (Bd)--(Wd1)
};
\draw[densely dotted] {
(W11)--(W12) (W01)--(W02) (W21)--(W22)--(Bd) (Wd1)--(Wd2)
};
\draw[decorate,decoration={brace,amplitude=5},xshift=-2pt,yshift=-0pt]
(-2,1)--(-2,3) node [midway,xshift=-11pt,yshift=0pt]{\scriptsize $m_1$};
\draw[decorate,decoration={brace,amplitude=5},xshift=-2pt,yshift=-0pt]
(-2,5)--(-2,8) node [midway,xshift=-8pt,yshift=0pt]{\scriptsize $l$};
\draw[decorate,decoration={brace,amplitude=5},xshift=2pt,yshift=0pt]
(2,3)--(2,1) node [midway,xshift=12pt,yshift=0pt]{\scriptsize $m_2$};
\draw[decorate,decoration={brace,amplitude=5},xshift=2pt,yshift=0pt]
(2,8)--(2,6) node [midway,xshift=12pt,yshift=0pt]{\scriptsize $m_d$};
\end{tikzpicture}
\ \right]\,. 
\end{align}
By repeating similar computations, we have  
\begin{align}
[X]
&=\sum_{i=1}^{d-1}\sum_{a_i=0}^{k_i-1}(-1)^{k_1+\cdots+k_{i-1}+a_i}
[X_{i,a_i}]
+(-1)^{k_1+\cdots+k_{d-1}}[X_{d,0}]\,,
\end{align}
where
\[
X_{i,a}
=
\begin{tikzpicture}[thick,x=8pt,y=8pt,baseline=50pt]
\node[vertex_black] (Bi) at (0,0) {};
\node[vertex_white] (Wi-11) at (-2,1) {};
\node[vertex_white] (Wi-12) at (-2,3) {};
\node[vertex_black] (B2) at (-2,5) {}; 
\node[vertex_white] (W11) at (-2,6) {};
\node[vertex_white] (W12) at (-2,8) {};
\node[vertex_black] (B1) at (-2,9) {};
\node[vertex_white] (W01) at (-2,10) {}; 
\node[vertex_white] (W02) at (-2,13) {};
\node[vertex_white] (Wi1) at (2,1) {};
\node[vertex_white] (Wi2) at (2,3) {}; 
\node[vertex_black] (Bi+1) at (11,0) {};
\node[vertex_white] (W1) at (9,1) {};
\node[vertex_white] (W2) at (9,3) {};
\node[vertex_white] (Wi+11) at (13,1) {};
\node[vertex_white] (Wi+12) at (13,3) {};
\node[vertex_black] (Bd) at (13,5) {};
\node[vertex_white] (Wd1) at (13,6) {};
\node[vertex_white] (Wd2) at (13,8) {};
\draw{
(Bi)--(Wi-11) (B2)--(W11) (W12)--(B1)--(W01) (Bi)--(Wi1) 
(Bi+1)--(W1) (Bi+1)--(Wi+11) (Bd)--(Wd1)
};
\draw[densely dotted] {
(Wi-11)--(Wi-12)--(B2) (W11)--(W12) (W01)--(W02) (Wi1)--(Wi2) 
(W1)--(W2) (Wi+11)--(Wi+12)--(Bd) (Wd1)--(Wd2)
};
\draw[decorate,decoration={brace,amplitude=5},xshift=-2pt,yshift=0pt]
(-2,10)--(-2,13) node [midway,xshift=-10pt,yshift=0pt]{\scriptsize $l$};
\draw[decorate,decoration={brace,amplitude=5},xshift=-2pt,yshift=0pt]
(-2,6)--(-2,8) node [midway,xshift=-12pt,yshift=0pt]{\scriptsize $m_1$};
\draw[decorate,decoration={brace,amplitude=5},xshift=-2pt,yshift=0pt]
(-2,1)--(-2,3) node [midway,xshift=-18pt,yshift=0pt]{\scriptsize $m_{i-1}$};
\draw[decorate,decoration={brace,amplitude=5},xshift=2pt,yshift=0pt]
(2,3)--(2,1) node [midway,xshift=20pt,yshift=0pt]{\scriptsize $m_i\!-\!a$};
\draw[decorate,decoration={brace,amplitude=5},xshift=-2pt,yshift=0pt]
(9,1)--(9,3) node [midway,xshift=-10pt,yshift=0pt]{\scriptsize $a$};
\draw[decorate,decoration={brace,amplitude=3},xshift=2pt,yshift=0pt]
(13,3)--(13,1) node [midway,xshift=16pt,yshift=0pt]{\scriptsize $m_{i+1}$};
\draw[decorate,decoration={brace,amplitude=3},xshift=2pt,yshift=0pt]
(13,8)--(13,6) node [midway,xshift=12pt,yshift=0pt]{\scriptsize $m_d$};
\end{tikzpicture}
\,, \quad
X_{d,0}
=
\begin{tikzpicture}[thick,x=8pt,y=8pt,baseline=50pt]
\node[vertex_black] (Bd) at (0,0) {};
\node[vertex_white] (Wd-11) at (-2,1) {};
\node[vertex_white] (Wd-12) at (-2,3) {};
\node[vertex_black] (B2) at (-2,5) {}; 
\node[vertex_white] (W11) at (-2,6) {};
\node[vertex_white] (W12) at (-2,8) {};
\node[vertex_black] (B1) at (-2,9) {};
\node[vertex_white] (W01) at (-2,10) {}; 
\node[vertex_white] (W02) at (-2,13) {};
\node[vertex_white] (Wd1) at (2,1) {};
\node[vertex_white] (Wd2) at (2,3) {}; 
\draw{
(Bd)--(Wd-11) (B2)--(W11) (W12)--(B1)--(W01) (Bd)--(Wd1) 
};
\draw[densely dotted] {
(Wd-11)--(Wd-12)--(B2) (W11)--(W12) (W01)--(W02) (Wd1)--(Wd2)
};
\draw[decorate,decoration={brace,amplitude=5},xshift=-2pt,yshift=0pt]
(-2,10)--(-2,13) node [midway,xshift=-10pt,yshift=0pt]{\scriptsize $l$};
\draw[decorate,decoration={brace,amplitude=5},xshift=-2pt,yshift=0pt]
(-2,6)--(-2,8) node [midway,xshift=-12pt,yshift=0pt]{\scriptsize $m_1$};
\draw[decorate,decoration={brace,amplitude=5},xshift=-2pt,yshift=0pt]
(-2,1)--(-2,3) node [midway,xshift=-18pt,yshift=0pt]{\scriptsize $m_{d-1}$};
\draw[decorate,decoration={brace,amplitude=5},xshift=2pt,yshift=0pt]
(2,3)--(2,1) node [midway,xshift=12pt,yshift=0pt]{\scriptsize $m_d$};
\end{tikzpicture}
\,,
\]
that is, 
\begin{align*}
X_{i,a}
&=X_{k_i-1-a}(k_{i-1},\ldots,k_1,l+1)\sqcup X_{a}(k_{i+1},\ldots,k_d)\,, \\
X_{d,0}
&=X_{k_d-1}(k_{d-1},\ldots,k_1,l+1)\,.
\end{align*}
Notice that $X_{d,0}$ is always admissible because $l>0$.
By taking the admissible parts and making the integrals associated with these 2-posets, we have 
\begin{equation}
\label{for:Pl}
\begin{aligned}
 P_l(\bk)
&=\sum_{i=1}^{d-1}\sum_{a_i=0}^{k_i-1}(-1)^{k_1+\cdots+k_{i-1}+a_i}
I([X_{i,a_i}]^{\adm})\\
&\quad +(-1)^{k_1+\cdots+k_{d-1}}P_{k_d-1}(k_{d-1},\ldots,k_1,l+1)\,.
\end{aligned}
\end{equation}

If $m_d>0$ (i.e., $\bk$ is admissible),
since $X_{i,a}$ is also admissible, 
the formula \eqref{eq:general weighted sum adm} is immediately obtained from \eqref{for:P is I of X} and \eqref{for:Pl}.

If $m_d=0$ (i.e., $\bk$ is non-admissible),
noticing that $X_{i,a}$ is admissible if and only if 
$i=d-1$ and $a>0$, we have\footnote{For a condition $P$, we let $\mathbbm{1}_{P}$ denote the indicator function on $P$, that is, $\mathbbm{1}_{P}=1$ if $P$ is satisfied and $0$ otherwise. 
We also put $\overline{\mathbbm{1}}_{P}=1-\mathbbm{1}_{P}$.
Condition with multiple lines
stands for the conjunction of all lines.} from \eqref{for:Pl}
\begin{equation}
\begin{aligned}
&P_l(\bk)
=\sum_{i=1}^{d-1}\sum_{a_i=0}^{k_i-1}
\overline{\mathbbm{1}}_{\substack{i=d-1\\ a_i\ne 0}}
(-1)^{k_1+\cdots+k_{i-1}+a_i}
I([X_{i,a_i}]^{\adm})\\
&\quad +\sum^{d}_{i=d-1}\sum_{a_{i}=0}^{k_{i}-1}
\overline{\mathbbm{1}}_{\substack{i=d-1\\ a_i=0}}
(-1)^{k_1+\cdots+k_{i-1}+a_{i}}P_{k_{i}-1-a_i}(k_{i-1},\ldots,k_{1},l+1)P_{a_i}(k_{i+1},\ldots,k_{d-1},1)\,.
\end{aligned}
\end{equation}
Now, we compute $I([X_{i,a}]^{\adm})$. 
Observe that 
\[
[X_{i,a}]^\adm
=[X_{i,a}^{(1)}]+[X_{i,a}^{(2)}]+[X_{i,a}^{(3)}]\,, 
\]
where 
\begin{align*}
X_{i,a}^{(1)}
&=
\begin{tikzpicture}[thick,x=8pt,y=8pt,baseline=50pt]
\node[vertex_black] (Bi) at (0,0) {};
\node[vertex_white] (Wi-11) at (-2,1) {};
\node[vertex_white] (Wi-12) at (-2,3) {};
\node[vertex_black] (B2) at (-2,5) {}; 
\node[vertex_white] (W11) at (-2,6) {};
\node[vertex_white] (W12) at (-2,8) {};
\node[vertex_black] (B1) at (-2,9) {};
\node[vertex_white] (W01) at (-2,10) {}; 
\node[vertex_white] (W02) at (-2,13) {};
\node[vertex_white] (Wi1) at (2,1) {};
\node[vertex_white] (Wi2) at (2,3) {}; 
\node[vertex_black] (Bi+1) at (11,0) {};
\node[vertex_white] (W1) at (9,1) {};
\node[vertex_white] (W2) at (9,3) {};
\node[vertex_white] (Wi+11) at (13,1) {};
\node[vertex_white] (Wi+12) at (13,3) {};
\node[vertex_black] (Bd-1) at (13,5) {};
\node[vertex_white] (Wd-11) at (13,6) {};
\node[vertex_white] (Wd-12) at (13,8) {};
\node[vertex_black] (Bd) at (13,9) {};
\draw{
(Bi)--(Wi-11) (B2)--(W11) (W12)--(B1)--(W01) (Bi)--(Wi1) (Wi2)--(W02)
(Bi+1)--(W1) (Bi+1)--(Wi+11) (Bd-1)--(Wd-11) (Wd-12)--(Bd)--(W2) (Bd)--(W02) 
};
\draw[densely dotted] {
(Wi-11)--(Wi-12)--(B2) (W11)--(W12) (W01)--(W02) (Wi1)--(Wi2) 
(W1)--(W2) (Wi+11)--(Wi+12)--(Bd-1) (Wd-11)--(Wd-12)
};
\draw[decorate,decoration={brace,amplitude=5},xshift=-2pt,yshift=0pt]
(-2,10)--(-2,13) node [midway,xshift=-10pt,yshift=0pt]{\scriptsize $l$};
\draw[decorate,decoration={brace,amplitude=5},xshift=-2pt,yshift=0pt]
(-2,6)--(-2,8) node [midway,xshift=-12pt,yshift=0pt]{\scriptsize $m_1$};
\draw[decorate,decoration={brace,amplitude=5},xshift=-2pt,yshift=0pt]
(-2,1)--(-2,3) node [midway,xshift=-18pt,yshift=0pt]{\scriptsize $m_{i-1}$};
\draw[decorate,decoration={brace,amplitude=5},xshift=2pt,yshift=0pt]
(2,3)--(2,1) node [midway,xshift=20pt,yshift=0pt]{\scriptsize $m_i-a$};
\draw[decorate,decoration={brace,amplitude=5},xshift=-2pt,yshift=0pt]
(9,1)--(9,3) node [midway,xshift=-10pt,yshift=0pt]{\scriptsize $a$};
\draw[decorate,decoration={brace,amplitude=3},xshift=2pt,yshift=0pt]
(13,3)--(13,1) node [midway,xshift=16pt,yshift=0pt]{\scriptsize $m_{i+1}$};
\draw[decorate,decoration={brace,amplitude=3},xshift=2pt,yshift=0pt]
(13,8)--(13,6) node [midway,xshift=16pt,yshift=0pt]{\scriptsize $m_{d-1}$};
\end{tikzpicture}
\,,\\
X_{i,a}^{(2)}
&=
\begin{tikzpicture}[thick,x=8pt,y=8pt,baseline=50pt]
\node[vertex_black] (Bi) at (0,0) {};
\node[vertex_white] (Wi-11) at (-2,1) {};
\node[vertex_white] (Wi-12) at (-2,3) {};
\node[vertex_black] (B2) at (-2,5) {}; 
\node[vertex_white] (W11) at (-2,6) {};
\node[vertex_white] (W12) at (-2,8) {};
\node[vertex_black] (B1) at (-2,9) {};
\node[vertex_white] (W01) at (-2,10) {}; 
\node[vertex_white] (W02) at (-2,13) {};
\node[vertex_white] (Wi1) at (2,1) {};
\node[vertex_white] (Wi2) at (2,14) {}; 
\node[vertex_black] (Bi+1) at (11,0) {};
\node[vertex_white] (W1) at (9,1) {};
\node[vertex_white] (W2) at (9,3) {};
\node[vertex_white] (Wi+11) at (13,1) {};
\node[vertex_white] (Wi+12) at (13,3) {};
\node[vertex_black] (Bd-1) at (13,5) {};
\node[vertex_white] (Wd-11) at (13,6) {};
\node[vertex_white] (Wd-12) at (13,8) {};
\node[vertex_black] (Bd) at (13,9) {};
\draw{
(Bi)--(Wi-11) (B2)--(W11) (W12)--(B1)--(W01) (Bi)--(Wi1) (Wi2)--(W02)
(Bi+1)--(W1) (Bi+1)--(Wi+11) (Bd-1)--(Wd-11) (Wd-12)--(Bd)--(W2) (Bd)--(Wi2) 
};
\draw[densely dotted] {
(Wi-11)--(Wi-12)--(B2) (W11)--(W12) (W01)--(W02) (Wi1)--(Wi2) 
(W1)--(W2) (Wi+11)--(Wi+12)--(Bd-1) (Wd-11)--(Wd-12)
};
\draw[decorate,decoration={brace,amplitude=5},xshift=-2pt,yshift=0pt]
(-2,10)--(-2,13) node [midway,xshift=-10pt,yshift=0pt]{\scriptsize $l$};
\draw[decorate,decoration={brace,amplitude=5},xshift=-2pt,yshift=0pt]
(-2,6)--(-2,8) node [midway,xshift=-12pt,yshift=0pt]{\scriptsize $m_1$};
\draw[decorate,decoration={brace,amplitude=5},xshift=-2pt,yshift=0pt]
(-2,1)--(-2,3) node [midway,xshift=-18pt,yshift=0pt]{\scriptsize $m_{i-1}$};
\draw[decorate,decoration={brace,amplitude=5},xshift=2pt,yshift=0pt]
(2,14)--(2,1) node [midway,xshift=20pt,yshift=0pt]{\scriptsize $m_i-a$};
\draw[decorate,decoration={brace,amplitude=5},xshift=-2pt,yshift=0pt]
(9,1)--(9,3) node [midway,xshift=-10pt,yshift=0pt]{\scriptsize $a$};
\draw[decorate,decoration={brace,amplitude=3},xshift=2pt,yshift=0pt]
(13,3)--(13,1) node [midway,xshift=16pt,yshift=0pt]{\scriptsize $m_{i+1}$};
\draw[decorate,decoration={brace,amplitude=3},xshift=2pt,yshift=0pt]
(13,8)--(13,6) node [midway,xshift=16pt,yshift=0pt]{\scriptsize $m_{d-1}$};
\end{tikzpicture}
\,,
\intertext{and}
X_{i,a}^{(3)}
&=
\begin{tikzpicture}[thick,x=8pt,y=8pt,baseline=50pt]
\node[vertex_black] (Bi) at (0,0) {};
\node[vertex_white] (Wi-11) at (-2,1) {};
\node[vertex_white] (Wi-12) at (-2,3) {};
\node[vertex_black] (B2) at (-2,5) {}; 
\node[vertex_white] (W11) at (-2,6) {};
\node[vertex_white] (W12) at (-2,8) {};
\node[vertex_black] (B1) at (-2,9) {};
\node[vertex_white] (W01) at (-2,10) {}; 
\node[vertex_white] (W02) at (-2,13) {};
\node[vertex_white] (Wi1) at (2,1) {};
\node[vertex_white] (Wi2) at (2,3) {}; 
\node[vertex_black] (Bi+1) at (11,0) {};
\node[vertex_white] (W1) at (9,1) {};
\node[vertex_white] (W2) at (9,11) {};
\node[vertex_white] (Wi+11) at (13,1) {};
\node[vertex_white] (Wi+12) at (13,3) {};
\node[vertex_black] (Bd-1) at (13,5) {};
\node[vertex_white] (Wd-11) at (13,6) {};
\node[vertex_white] (Wd-12) at (13,8) {};
\node[vertex_black] (Bd) at (13,9) {};
\draw{
(Bi)--(Wi-11) (B2)--(W11) (W12)--(B1)--(W01) (Bi)--(Wi1) 
(Bi+1)--(W1) (Bi+1)--(Wi+11) (Bd-1)--(Wd-11) (Wd-12)--(Bd)--(W2)
};
\draw[densely dotted] {
(Wi-11)--(Wi-12)--(B2) (W11)--(W12) (W01)--(W02) (Wi1)--(Wi2) 
(W1)--(W2) (Wi+11)--(Wi+12)--(Bd-1) (Wd-11)--(Wd-12)
};
\draw[decorate,decoration={brace,amplitude=5},xshift=-2pt,yshift=0pt]
(-2,10)--(-2,13) node [midway,xshift=-10pt,yshift=0pt]{\scriptsize $l$};
\draw[decorate,decoration={brace,amplitude=5},xshift=-2pt,yshift=0pt]
(-2,6)--(-2,8) node [midway,xshift=-12pt,yshift=0pt]{\scriptsize $m_1$};
\draw[decorate,decoration={brace,amplitude=5},xshift=-2pt,yshift=0pt]
(-2,1)--(-2,3) node [midway,xshift=-18pt,yshift=0pt]{\scriptsize $m_{i-1}$};
\draw[decorate,decoration={brace,amplitude=5},xshift=2pt,yshift=0pt]
(2,3)--(2,1) node [midway,xshift=20pt,yshift=0pt]{\scriptsize $m_i-a$};
\draw[decorate,decoration={brace,amplitude=5},xshift=-2pt,yshift=0pt]
(9,1)--(9,11) node [midway,xshift=-10pt,yshift=0pt]{\scriptsize $a$};
\draw[decorate,decoration={brace,amplitude=3},xshift=2pt,yshift=0pt]
(13,3)--(13,1) node [midway,xshift=16pt,yshift=0pt]{\scriptsize $m_{i+1}$};
\draw[decorate,decoration={brace,amplitude=3},xshift=2pt,yshift=0pt]
(13,8)--(13,6) node [midway,xshift=16pt,yshift=0pt]{\scriptsize $m_{d-1}$};
\end{tikzpicture}
\,. 
\end{align*}
Here we understand that $X_{i,m_i}^{(2)}=X_{i,0}^{(3)}=0$. 

It is easy from \eqref{for:P is I of X} again to see that 
\[
 I([X_{i,a}^{(3)}])
=P_{k_i-1-a}(k_{i-1},\ldots,k_1,l+1)\,P_a(k_{i+1},\ldots,k_{d-1},1)
\]
 and hence  
\begin{equation}
\label{for:Plnonad2}
\begin{aligned}
& P_l(\bk)
=\sum_{i=1}^{d-1}\sum_{a_i=0}^{k_i-1}\overline{\mathbbm{1}}_{\substack{i=d-1\\ a_i\ne 0}}(-1)^{k_1+\cdots+k_{i-1}+a_i}
I([X^{(1)}_{i,a_i}]+[X^{(2)}_{i,a_i}])\\
&\quad +\sum^{d}_{i=1}\sum_{a_{i}=0}^{k_{i}-1}
(-1)^{k_1+\cdots+k_{i-1}+a_{i}}P_{k_{i}-1-a_i}(k_{i-1},\ldots,k_{1},l+1)P_{a_i}(k_{i+1},\ldots,k_{d-1},1)\,.
\end{aligned}
\end{equation}
Moreover, we see that $[X_{i,a}^{(1)}]+[X_{i,a}^{(2)}]$ can be written in terms of 2-posets
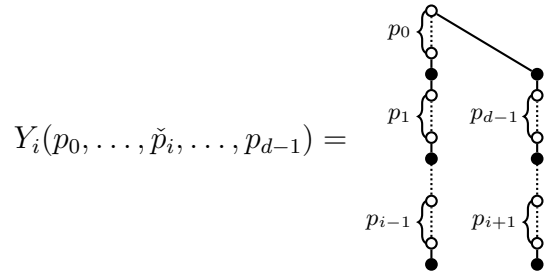
\begin{equation}
\label{def:Y}
Y_i(p_0,\ldots,\check{p}_i,\ldots,p_{d-1})
=
\begin{tikzpicture}[thick,x=8pt,y=8pt,baseline=44pt]
\node[vertex_black] (Bi) at (0,0) {};
\node[vertex_white] (Wi-11) at (0,1) {};
\node[vertex_white] (Wi-12) at (0,3) {};
\node[vertex_black] (B2) at (0,5) {}; 
\node[vertex_white] (W11) at (0,6) {};
\node[vertex_white] (W12) at (0,8) {};
\node[vertex_black] (B1) at (0,9) {};
\node[vertex_white] (W01) at (0,10) {}; 
\node[vertex_white] (W02) at (0,12) {};
\node[vertex_black] (Bi+1) at (5,0) {};
\node[vertex_white] (Wi+11) at (5,1) {}; 
\node[vertex_white] (Wi+12) at (5,3) {};
\node[vertex_black] (Bd-1) at (5,5) {}; 
\node[vertex_white] (Wd-11) at (5,6) {};
\node[vertex_white] (Wd-12) at (5,8) {};
\node[vertex_black] (Bd) at (5,9) {};
\draw{
(Bi)--(Wi-11) (B2)--(W11) (W12)--(B1)--(W01) 
(Bi+1)--(Wi+11) (Bd-1)--(Wd-11) (Wd-12)--(Bd)--(W02)
};
\draw[densely dotted] {
(Wi-11)--(Wi-12)--(B2) (W11)--(W12) (W01)--(W02) 
(Wi+11)--(Wi+12)--(Bd-1) (Wd-11)--(Wd-12) 
};
\draw[decorate,decoration={brace,amplitude=3},xshift=-2pt,yshift=0pt]
(0,1)--(0,3) node [midway,xshift=-14pt,yshift=0pt]{\scriptsize $p_{i-1}$};
\draw[decorate,decoration={brace,amplitude=3},xshift=-2pt,yshift=0pt]
(0,6)--(0,8) node [midway,xshift=-10pt,yshift=0pt]{\scriptsize $p_1$};
\draw[decorate,decoration={brace,amplitude=3},xshift=-2pt,yshift=0pt]
(0,10)--(0,12) node [midway,xshift=-10pt,yshift=0pt]{\scriptsize $p_0$};
\draw[decorate,decoration={brace,amplitude=3},xshift=-2pt,yshift=0pt]
(5,1)--(5,3) node [midway,xshift=-14pt,yshift=0pt]{\scriptsize $p_{i+1}$};
\draw[decorate,decoration={brace,amplitude=3},xshift=-2pt,yshift=0pt]
(5,6)--(5,8) node [midway,xshift=-14pt,yshift=0pt]{\scriptsize $p_{d-1}$};
\end{tikzpicture}
\end{equation}
for various values $p_0,,\ldots,\check{p}_i,\ldots,p_{d-1}$ 
($\check{p}_i$ means that $p_i$ is skipped). 
Actually, setting $m_0=l$, we have   
\begin{align}
 [X^{(1)}_{i,a}]
&=\sum_{\substack{\bb_i=(b_0,\ldots,\check{b}_{i},\ldots,b_{d-1})\in\mathbb{Z}^{d-1}_{\ge 0} \\
\wt(b_0,\ldots,b_{i-1})=m_i-a_i \\ 
\wt(b_{i+1},\ldots,b_{d-1})=a_i}}
\binom{m_0+b_0-1}{m_0-1}\prod^{d-1}_{\substack{j=1 \\ j\ne i}}\binom{m_j+b_j}{m_j} [Y_i(\bb_i+\bmm_i)]\,,\\
 [X^{(2)}_{i,a}]
&=\sum_{\substack{\bb_i=(b_0,\ldots,\check{b}_{i},\ldots,b_{d-1})\in\mathbb{Z}^{d-1}_{\ge 0} \\ \wt(b_0,\ldots,b_{i-1})=m_i-a_i \\ 
\wt(b_{i+1},\ldots,b_{d-1})=a_i}}
\binom{m_0+b_0-1}{m_0}\prod^{d-1}_{\substack{j=1 \\ j\ne i}}\binom{m_j+b_j}{m_j} [Y_i(\bb_i+\bmm_i)]\,,
\end{align}
where $\bmm_i=(m_0,\ldots,\check{m}_i,\ldots,m_{d-1})$.
Hence, using the identity
$\binom{s-1}{t-1}+\binom{s-1}{t}=\binom{s}{t}$ for $s,t\ge 0$,
we have 
\begin{equation}
\label{for:X1X2}    
\begin{aligned}
[X^{(1)}_{i,a}]+[X^{(2)}_{i,a}]
&=\sum_{\substack{\bb_i=(b_0,\ldots,\check{b}_{i},\ldots,b_{d-1})\in\mathbb{Z}^{d-1}_{\ge 0} \\ \wt(b_0,\ldots,b_{i-1})=m_i-a_i \\ 
\wt(b_{i+1},\ldots,b_{d-1})=a_i}}
\prod^{d-1}_{\substack{j=0 \\ j\ne i}}\binom{m_j+b_j}{m_j} [Y_i(\bb_i+\bmm_i)]\,.
\end{aligned} 
\end{equation}
Substituting this into \eqref{for:Plnonad2} and changing the order of summations, we see that  
\begin{align*}
& P_l(\bk)
=\sum_{i=1}^{d-1}
(-1)^{l+k_i}
\sum_{\substack{\bb_i=(b_0,\ldots,\check{b}_{i},\ldots,b_{d-1})\in\mathbb{Z}^{d-1}_{\ge 1}\\
\wt(\bb_i)=\wt(\bk)+l-1}}
(-1)^{b_0+\cdots+b_{i-1}}
\prod^{d-1}_{\substack{j=0 \\ j\ne i}}\binom{b_j-1}{m_j} 
 I[Y_i(\bb_i-\{1\}^{d-1})]\\
&\quad +\sum^{d}_{i=1}\sum_{a_{i}=0}^{k_{i}-1}
(-1)^{k_1+\cdots+k_{i-1}+a_{i}}P_{k_{i}-1-a_i}(k_{i-1},\ldots,k_{1},l+1)P_{a_i}(k_{i+1},\ldots,k_{d-1},1)\,.
\end{align*}
Therefore, one obtains \eqref{eq:general weighted sum nonadm} by employing the expression of $I([Y_i(\bb_i-\{1\}^{d-1})])$ given in Lemma \ref{lem:I of Y} below.
This completes the proof.
\end{proof}

\begin{lemma}
\label{lem:I of Y}
For $1\le i\le d-1$ and $\bp=(p_0,\ldots,\check{p}_i,\ldots,p_{d-1})\in\Z_{\ge 0}^{d-1}$ 
with $p_0>0$, we have   
\begin{equation}
\begin{aligned}
I([Y_i(\bp)])
&=\sum_{j=i}^{d-1}\sum_{c_j=0}^{p_j-1}
(-1)^{c_j+p_{j+1}+\cdots+p_{d-1}}
\zeta\begin{varray}
c_j+1,p_{j+1}+1,\cdots,p_{d-1}+1\\
p_{i-1}+1,\cdots,p_1+1,p_0+1
\end{varray} \\
&\quad \times \zeta(p_{i+1}+1,\ldots,p_{j-1}+1,p_{j}-c_{j}+1)\,,
\end{aligned} 
\end{equation}
where we set $p_i=1$. 
\end{lemma}
\begin{proof}
When $i=d-1$, the formula reads as 
\[I\left(
\begin{tikzpicture}[thick,x=8pt,y=8pt,baseline=44pt]
\node[vertex_black] (Bi) at (0,0) {};
\node[vertex_white] (Wi-11) at (0,1) {};
\node[vertex_white] (Wi-12) at (0,3) {};
\node[vertex_black] (B2) at (0,5) {}; 
\node[vertex_white] (W11) at (0,6) {};
\node[vertex_white] (W12) at (0,8) {};
\node[vertex_black] (B1) at (0,9) {};
\node[vertex_white] (W01) at (0,10) {}; 
\node[vertex_white] (W02) at (0,12) {};
\node[vertex_black] (Bd) at (3,9) {};
\draw{
(Bi)--(Wi-11) (B2)--(W11) (W12)--(B1)--(W01) 
(Bd)--(W02)
};
\draw[densely dotted] {
(Wi-11)--(Wi-12)--(B2) (W11)--(W12) (W01)--(W02) 
};
\draw[decorate,decoration={brace,amplitude=3},xshift=-2pt,yshift=0pt]
(0,1)--(0,3) node [midway,xshift=-14pt,yshift=0pt]{\scriptsize $p_{d-2}$};
\draw[decorate,decoration={brace,amplitude=3},xshift=-2pt,yshift=0pt]
(0,6)--(0,8) node [midway,xshift=-10pt,yshift=0pt]{\scriptsize $p_1$};
\draw[decorate,decoration={brace,amplitude=3},xshift=-2pt,yshift=0pt]
(0,10)--(0,12) node [midway,xshift=-10pt,yshift=0pt]{\scriptsize $p_0$};
\end{tikzpicture}
\ \right)
=\zeta\begin{varray}
1 \\
p_{d-2}+1,\ldots,p_1+1,p_0+1
\end{varray}, 
\]
which is a special case of Theorem \ref{thm:integral series identity}. 
For $i<d-1$, we compute as  
\begin{align}
[Y_i(\bp)]
&=\sum_{c_{d-1}=0}^{p_{d-1}-1}(-1)^{c_{d-1}}
\left[
\begin{tikzpicture}[thick,x=8pt,y=8pt,baseline=44pt]
\node[vertex_black] (Bi) at (0,0) {};
\node[vertex_white] (Wi-11) at (0,1) {};
\node[vertex_white] (Wi-12) at (0,3) {};
\node[vertex_black] (B2) at (0,5) {}; 
\node[vertex_white] (W11) at (0,6) {};
\node[vertex_white] (W12) at (0,8) {};
\node[vertex_black] (B1) at (0,9) {};
\node[vertex_white] (W01) at (0,10) {}; 
\node[vertex_white] (W02) at (0,12) {};
\node[vertex_black] (Bd) at (2,9) {};
\node[vertex_white] (Wd1) at (3,10) {}; 
\node[vertex_white] (Wd2) at (5,12) {}; 
\node[vertex_black] (Bi+1) at (8,0) {};
\node[vertex_white] (Wi+11) at (8,1) {}; 
\node[vertex_white] (Wi+12) at (8,3) {};
\node[vertex_black] (Bd-1) at (8,5) {}; 
\node[vertex_white] (Wd-11) at (8,6) {};
\node[vertex_white] (Wd-12) at (8,8) {};
\draw{
(Bi)--(Wi-11) (B2)--(W11) (W12)--(B1)--(W01) (W02)--(Bd)--(Wd1)
(Bi+1)--(Wi+11) (Bd-1)--(Wd-11) 
};
\draw[densely dotted] {
(Wi-11)--(Wi-12)--(B2) (W11)--(W12) (W01)--(W02) (Wd1)--(Wd2)
(Wi+11)--(Wi+12)--(Bd-1) (Wd-11)--(Wd-12) 
};
\draw[decorate,decoration={brace,amplitude=3},xshift=-2pt,yshift=0pt]
(0,1)--(0,3) node [midway,xshift=-14pt,yshift=0pt]{\scriptsize $p_{i-1}$};
\draw[decorate,decoration={brace,amplitude=3},xshift=-2pt,yshift=0pt]
(0,6)--(0,8) node [midway,xshift=-10pt,yshift=0pt]{\scriptsize $p_1$};
\draw[decorate,decoration={brace,amplitude=3},xshift=-2pt,yshift=0pt]
(0,10)--(0,12) node [midway,xshift=-10pt,yshift=0pt]{\scriptsize $p_0$};
\draw[decorate,decoration={brace,amplitude=3},xshift=-2pt,yshift=2pt]
(3,10)--(5,12) node [midway,xshift=-14pt,yshift=4pt]{\scriptsize $c_{d-1}$};
\draw[decorate,decoration={brace,amplitude=3},xshift=-2pt,yshift=0pt]
(8,1)--(8,3) node [midway,xshift=-14pt,yshift=0pt]{\scriptsize $p_{i+1}$};
\draw[decorate,decoration={brace,amplitude=3},xshift=-2pt,yshift=0pt]
(8,6)--(8,8) node [midway,xshift=-28pt,yshift=0pt]{\scriptsize $p_{d-1}-c_{d-1}$};
\end{tikzpicture}
\ \right]
+(-1)^{p_{d-1}}
\left[
\begin{tikzpicture}[thick,x=8pt,y=8pt,baseline=44pt]
\node[vertex_black] (Bi) at (0,0) {};
\node[vertex_white] (Wi-11) at (0,1) {};
\node[vertex_white] (Wi-12) at (0,3) {};
\node[vertex_black] (B2) at (0,5) {}; 
\node[vertex_white] (W11) at (0,6) {};
\node[vertex_white] (W12) at (0,8) {};
\node[vertex_black] (B1) at (0,9) {};
\node[vertex_white] (W01) at (0,10) {}; 
\node[vertex_white] (W02) at (0,12) {};
\node[vertex_black] (Bd) at (2,9) {};
\node[vertex_white] (Wd1) at (3,10) {}; 
\node[vertex_white] (Wd2) at (5,12) {}; 
\node[vertex_black] (Bi+1) at (7,0) {};
\node[vertex_white] (Wi+11) at (7,1) {}; 
\node[vertex_white] (Wi+12) at (7,3) {};
\node[vertex_black] (Bd-2) at (7,5) {}; 
\node[vertex_white] (Wd-21) at (7,6) {}; 
\node[vertex_white] (Wd-22) at (7,8) {};
\node[vertex_black] (Bd-1) at (7,9) {}; 
\draw{
(Bi)--(Wi-11) (B2)--(W11) (W12)--(B1)--(W01) (W02)--(Bd)--(Wd1)
(Bi+1)--(Wi+11) (Bd-2)--(Wd-21) (Wd-22)--(Bd-1)--(Wd2) 
};
\draw[densely dotted] {
(Wi-11)--(Wi-12)--(B2) (W11)--(W12) (W01)--(W02) (Wd1)--(Wd2)
(Wi+11)--(Wi+12)--(Bd-1) 
};
\draw[decorate,decoration={brace,amplitude=3},xshift=-2pt,yshift=0pt]
(0,1)--(0,3) node [midway,xshift=-14pt,yshift=0pt]{\scriptsize $p_{i-1}$};
\draw[decorate,decoration={brace,amplitude=3},xshift=-2pt,yshift=0pt]
(0,6)--(0,8) node [midway,xshift=-10pt,yshift=0pt]{\scriptsize $p_1$};
\draw[decorate,decoration={brace,amplitude=3},xshift=-2pt,yshift=0pt]
(0,10)--(0,12) node [midway,xshift=-10pt,yshift=0pt]{\scriptsize $p_0$};
\draw[decorate,decoration={brace,amplitude=3},xshift=-2pt,yshift=2pt]
(3,10)--(5,12) node [midway,xshift=-8pt,yshift=6pt]{\scriptsize $p_{d-1}$};
\draw[decorate,decoration={brace,amplitude=3},xshift=-2pt,yshift=0pt]
(7,1)--(7,3) node [midway,xshift=-14pt,yshift=0pt]{\scriptsize $p_{i+1}$};
\draw[decorate,decoration={brace,amplitude=3},xshift=-2pt,yshift=0pt]
(7,6)--(7,8) node [midway,xshift=-14pt,yshift=0pt]{\scriptsize $p_{d-2}$};
\end{tikzpicture}
\ \right] \\
&=\cdots\\
&=\sum_{j=i+1}^{d-1}\sum_{c_j=0}^{p_j-1}(-1)^{c_j+p_{j+1}+\cdots +p_{d-1}}
\left[
\begin{tikzpicture}[thick,x=8pt,y=8pt,baseline=44pt]
\node[vertex_black] (Bi) at (0,0) {};
\node[vertex_white] (Wi-11) at (0,1) {};
\node[vertex_white] (Wi-12) at (0,3) {};
\node[vertex_black] (B2) at (0,5) {}; 
\node[vertex_white] (W11) at (0,6) {};
\node[vertex_white] (W12) at (0,8) {};
\node[vertex_black] (B1) at (0,9) {};
\node[vertex_white] (W01) at (0,10) {}; 
\node[vertex_white] (W02) at (0,12) {};
\node[vertex_black] (Bd) at (2,9) {};
\node[vertex_white] (Wd1) at (3,10) {}; 
\node[vertex_white] (Wd2) at (5,12) {}; 
\node[vertex_black] (Bj+2) at (9,9) {};
\node[vertex_white] (Wj+21) at (10,10) {}; 
\node[vertex_white] (Wj+22) at (12,12) {}; 
\node[vertex_black] (Bj+1) at (14,9) {};
\node[vertex_white] (Wj+11) at (15,10) {}; 
\node[vertex_white] (Wj+12) at (17,12) {}; 
\node[vertex_black] (Bi+1) at (20,0) {};
\node[vertex_white] (Wi+11) at (20,1) {}; 
\node[vertex_white] (Wi+12) at (20,3) {};
\node[vertex_black] (Bj) at (20,5) {}; 
\node[vertex_white] (Wj1) at (20,6) {};
\node[vertex_white] (Wj2) at (20,8) {};
\draw{
(Bi)--(Wi-11) (B2)--(W11) (W12)--(B1)--(W01) (W02)--(Bd)--(Wd1) 
(Wd2)--(6,10.5) (8,10.5)--(Bj+2)--(Wj+21) (Wj+22)--(Bj+1)--(Wj+11) 
(Bi+1)--(Wi+11) (Bj)--(Wj1) 
};
\draw[densely dotted] {
(Wi-11)--(Wi-12)--(B2) (W11)--(W12) (W01)--(W02) (Wd1)--(Wd2) 
(6.3,10.5)--(7.7,10.5) (Wj+21)--(Wj+22) (Wj+11)--(Wj+12) 
(Wi+11)--(Wi+12)--(Bj) (Wj1)--(Wj2)
};
\draw[decorate,decoration={brace,amplitude=3},xshift=-2pt,yshift=0pt]
(0,1)--(0,3) node [midway,xshift=-14pt,yshift=0pt]{\scriptsize $p_{i-1}$};
\draw[decorate,decoration={brace,amplitude=3},xshift=-2pt,yshift=0pt]
(0,6)--(0,8) node [midway,xshift=-10pt,yshift=0pt]{\scriptsize $p_1$};
\draw[decorate,decoration={brace,amplitude=3},xshift=-2pt,yshift=0pt]
(0,10)--(0,12) node [midway,xshift=-10pt,yshift=0pt]{\scriptsize $p_0$};
\draw[decorate,decoration={brace,amplitude=3},xshift=-2pt,yshift=2pt]
(3,10)--(5,12) node [midway,xshift=-12pt,yshift=4pt]{\scriptsize $p_{d-1}$};
\draw[decorate,decoration={brace,amplitude=3},xshift=-2pt,yshift=2pt]
(10,10)--(12,12) node [midway,xshift=-12pt,yshift=4pt]{\scriptsize $p_{j+1}$};
\draw[decorate,decoration={brace,amplitude=3},xshift=-2pt,yshift=2pt]
(15,10)--(17,12) node [midway,xshift=-8pt,yshift=4pt]{\scriptsize $c_j$};
\draw[decorate,decoration={brace,amplitude=3},xshift=-2pt,yshift=0pt]
(20,1)--(20,3) node [midway,xshift=-14pt,yshift=0pt]{\scriptsize $p_{i+1}$};
\draw[decorate,decoration={brace,amplitude=3},xshift=-2pt,yshift=0pt]
(20,6)--(20,8) node [midway,xshift=-18pt,yshift=0pt]{\scriptsize $p_j-c_j$};
\end{tikzpicture}
\ \right]\\
&\qquad +(-1)^{p_{i+1}+\cdots+p_{d-1}}
\left[
\begin{tikzpicture}[thick,x=8pt,y=8pt,baseline=44pt]
\node[vertex_black] (Bi) at (0,0) {};
\node[vertex_white] (Wi-11) at (0,1) {};
\node[vertex_white] (Wi-12) at (0,3) {};
\node[vertex_black] (B2) at (0,5) {}; 
\node[vertex_white] (W11) at (0,6) {};
\node[vertex_white] (W12) at (0,8) {};
\node[vertex_black] (B1) at (0,9) {};
\node[vertex_white] (W01) at (0,10) {}; 
\node[vertex_white] (W02) at (0,12) {};
\node[vertex_black] (Bd) at (2,9) {};
\node[vertex_white] (Wd1) at (3,10) {}; 
\node[vertex_white] (Wd2) at (5,12) {}; 
\node[vertex_black] (Bi+2) at (9,9) {};
\node[vertex_white] (Wi+21) at (10,10) {}; 
\node[vertex_white] (Wi+22) at (12,12) {}; 
\node[vertex_black] (Bi+1) at (14,9) {};
\draw{
(Bi)--(Wi-11) (B2)--(W11) (W12)--(B1)--(W01) (W02)--(Bd)--(Wd1) 
(Wd2)--(6,10.5) (8,10.5)--(Bi+2)--(Wi+21) (Wi+22)--(Bi+1)
};
\draw[densely dotted] {
(Wi-11)--(Wi-12)--(B2) (W11)--(W12) (W01)--(W02) (Wd1)--(Wd2) 
(6.3,10.5)--(7.7,10.5) (Wi+21)--(Wi+22) 
};
\draw[decorate,decoration={brace,amplitude=3},xshift=-2pt,yshift=0pt]
(0,1)--(0,3) node [midway,xshift=-14pt,yshift=0pt]{\scriptsize $p_{i-1}$};
\draw[decorate,decoration={brace,amplitude=3},xshift=-2pt,yshift=0pt]
(0,6)--(0,8) node [midway,xshift=-10pt,yshift=0pt]{\scriptsize $p_1$};
\draw[decorate,decoration={brace,amplitude=3},xshift=-2pt,yshift=0pt]
(0,10)--(0,12) node [midway,xshift=-10pt,yshift=0pt]{\scriptsize $p_0$};
\draw[decorate,decoration={brace,amplitude=3},xshift=-2pt,yshift=2pt]
(3,10)--(5,12) node [midway,xshift=-12pt,yshift=4pt]{\scriptsize $p_{d-1}$};
\draw[decorate,decoration={brace,amplitude=3},xshift=-2pt,yshift=2pt]
(10,10)--(12,12) node [midway,xshift=-12pt,yshift=4pt]{\scriptsize $p_{i+1}$};
\end{tikzpicture}
\ \right]\,. 
\end{align}
This together with Theorem~\ref{thm:integral series identity} shows the desired result. 
\end{proof}

In the following we give explicit expressions of $P_l(\bk)$ for the cases $d=2$ and $d=3$ which are valid whether $\bk$ is admissible or not.
 
\begin{cor}\label{cor:pldep2}
For $k_1,k_2\ge 1$ and $l>0$, we have 
\begin{align}
\label{eq:d2 weighted sum any index}
&P_{l}(k_1,k_2)
=
(-1)^{k_2}
\sum_{\substack{w_1,w_2\ge2\\w_1+w_2=k_1+k_2+l}}
(-1)^{w_1}
\binom{w_1-1}{k_2-1}\binom{w_2-1}{l}\zeta(w_1)\zeta(w_2)\\
&\hspace{7mm}
+(-1)^{k_1}
\sum_{\substack{w_1\ge 1, w_2\ge2\\ w_1+w_2=k_1+k_2+l}}
\binom{w_1-1}{k_1-1}\binom{w_2-1}{l}\zeta(w_1,w_2)
+\mathbbm{1}_{k_2=1}\binom{l+k_1-1}{k_1-1}\zeta
\begin{varray}
1 \\ l+k_1 
\end{varray}
\,.
\end{align}
\end{cor}

\begin{cor}\label{cor:pldep3}
 For $k_1,k_2,k_3\ge 1$ and $l>0$, we have 
\begin{align}
& P_{l}(k_1,k_2,k_3)=(-1)^{k_1+k_{2}}
\!\!\!\!
\sum_{\substack{w_1,w_2\ge1,w_3\ge 2\\ w_1+w_2+w_3=k_1+k_2+k_3+l}}
\!\!\!\!
\binom{w_1-1}{k_2-1}\binom{w_2-1}{k_1-1}\binom{w_3-1}{l}
\zeta(w_1,w_2,w_3)\\
&\hspace{10mm}+(-1)^{k_2+k_3}\!\!\!\!\!\!\!
\!\!\!\!\!\!\! \sum_{\substack{
w_1\ge1, w_2\ge 2,w_3\ge 2\\
w_1+w_2+w_3=k_1+k_2+k_3+l
}}\!\!\!\!
(-1)^{w_1+w_2}
\binom{w_1-1}{k_2-1}\binom{w_2-1}{k_3-1}\binom{w_3-1}{l}
\zeta(w_1,w_2)\zeta(w_3)\\
&\hspace{10mm}
+(-1)^{k_1+k_3}
\!\!\!\!
\sum_{\substack{w_1\ge 2, w_2\ge1, w_3\ge2\\ w_1+w_2+w_3=k_1+k_2+k_3+l}}
\!\!\!\!
(-1)^{w_1}
\binom{w_1-1}{k_3-1}\binom{w_2-1}{k_1-1}\binom{w_3-1}{l}
\zeta(w_1)\zeta(w_2,w_3)\\
&\hspace{10mm}
+\mathbbm{1}_{k_3=1}(-1)^{k_2+1}\sum_{\substack{b_0\ge 2,b_2\ge 1\\ b_0+b_2=k_1+k_2+l}}(-1)^{b_2}
\binom{b_0-1}{l}\binom{b_2-1}{k_2-1}\\
&\hspace{60mm}
\times
\left\{
(-1)^{b_{2}}\zeta
\begin{varray}
1,b_2\\
b_0
\end{varray}
+\sum_{c_2=1}^{b_2-1}(-1)^{c_2}\zeta
\begin{varray}
c_2\\
b_0\end{varray}
\zeta(b_2-c_2+1)
\right\}
\\
&\hspace{10mm}
+\mathbbm{1}_{k_3=1}(-1)^{k_1} 
\sum_{\substack{b_0\ge 2,b_1\ge 1 \\ b_0+b_1=k_1+k_2+l}}
\binom{b_0-1}{l}\binom{b_1-1}{k_1-1}
\zeta
\begin{varray}
 1 \\ b_1,b_0
 \end{varray}
\,.
\end{align}
\end{cor}

\section{Ribbons}\label{sec:ribbons}

\subsection{Preparation}
In this section, we study the sums of Schur MZVs for ribbon diagrams. 
Recall that a skew Young diagram is called a \emph{ribbon} if it is connected and contains no $2\times 2$ block of boxes.
Explicitly, such a can be drawn as 
\begin{equation}\label{eq:ribbon}
\begin{tikzpicture}[x=12pt,y=12pt,baseline=90pt]
\draw (0,0) rectangle (4,1);
\draw[decorate,decoration={brace,amplitude=5},xshift=0pt,yshift=1pt]
(0,1)--(4,1) node [midway,xshift=0pt,yshift=10pt]{$s_1$};
\draw (4,0) rectangle  (5,4);
\draw[decorate,decoration={brace,mirror,amplitude=5},xshift=1pt,yshift=0pt]
(5,0)--(5,4) node [midway,xshift=12pt,yshift=0pt]{$r_1$};
\draw (4,4) rectangle (8,5);
\draw[decorate,decoration={brace,amplitude=5},xshift=0pt,yshift=1pt]
(4,5)--(8,5) node [midway,xshift=0pt,yshift=10pt]{$s_2$};
\draw (8,4) rectangle (9,8);
\draw[decorate,decoration={brace,mirror,amplitude=5},xshift=1pt,yshift=0pt]
(9,4)--(9,8) node [midway,xshift=12pt,yshift=0pt]{$r_2$};
\draw[dashed] (8.5,8.5)--(10.5,10.5);
\draw (10,11) rectangle (14,12);
\draw[decorate,decoration={brace,amplitude=5},xshift=0pt,yshift=1pt]
(10,12)--(14,12) node [midway,xshift=0pt,yshift=10pt]{$s_n$};
\draw (14,11) rectangle (15,15);
\draw[decorate,decoration={brace,mirror,amplitude=5},xshift=1pt,yshift=0pt]
(15,11)--(15,15) node [midway,xshift=12pt,yshift=0pt]{$r_n$};
\end{tikzpicture}
\end{equation}
where the integers $s_1\ge 0$, $s_2,\ldots,s_n,r_1,\ldots,r_n>0$ indicate the numbers of boxes. 

\begin{defn}
For integers $w,s_1,\ldots,s_n\ge 0$ and $r_1,\ldots,r_n>0$, 
we write 
\begin{equation}
S_w\binom{s_1,\ldots,s_n}{r_1,\ldots,r_n}\coloneqq\sum_{\substack{\bl_1,\ldots,\bl_n\\ \dep(\bl_i)=s_i\\ 
\bk_1,\ldots,\bk_n:\text{ adm.}\\ \dep(\bk_i)=r_i\\ \sum_i\wt(\bk_i)+\sum_i\wt(\bl_i)=w}}
\zeta\begin{varray}\bl_1,&\ldots,&\bl_n\\ \bk_1,&\ldots,&\bk_n \end{varray},
\end{equation}
where we define as a generalization of \eqref{eq:antihookzeta}
\begin{equation}\label{eq:zeta(ll/kk)}
\zeta\begin{varray}\bl_1,&\ldots,&\bl_n\\ \bk_1,&\ldots,&\bk_n \end{varray}
\coloneqq \sum_{\substack{0<b_{i,1}\le\cdots\le b_{i,s_i+1}\\ 0<a_{i,1}<\cdots<a_{i,r_i}\\ 
b_{i,s_i+1}=a_{i,r_i}\;(i=1,\ldots,n)\\
b_{i+1,1}<a_{i,1}\;(i=1,\ldots,n-1)}}
\prod_{i=1}^n\frac{1}{a_{i,1}^{k_{i,1}}\cdots a_{i,r_i}^{k_{i,r_i}} 
b_{i,1}^{l_{i,1}}\cdots b_{i,s_i}^{l_{i,s_i}}}
\end{equation}
for indices $\bl_i=(l_{i,1},\ldots,l_{i,s_i})$ of depth $s_i$ and 
$\bk_i=(k_{i,1},\ldots,k_{i,r_i})$ of depth $r_i$. 
Note that the series \eqref{eq:zeta(ll/kk)} is meaningful even if some $s_i$ is zero, 
i.e., $\bl_i=\varnothing$. 
\end{defn}

\begin{remark} Notice that only for $s_2,\dots,s_n>0$ the $S_w\binom{s_1,\ldots,s_n}{r_1,\ldots,r_n}$ gives the sum over all admissible tableaux of shape \eqref{eq:ribbon} as in \eqref{eq:sk}.
For example, we have 
\[S_w\begin{varray}2 \\ 2 \end{varray}=
\sum_{\substack{a,b,d\ge 1,\,c\ge 2\\ a+b+c+d=w}} 
\zeta\left(\;{\footnotesize\begin{ytableau}
\none & \none & d \\
a & b & c
\end{ytableau}\;}\right) = S_w\left(\;{\footnotesize\begin{ytableau}
\none & \none & \, \\
\, & \, & \,
\end{ytableau}\;}\right) \]
but 
\[S_w\begin{varray}2, & 0 \\ 1, & 1 \end{varray}=
\sum_{\substack{a,b\ge 1,\,c,d\ge 2\\ a+b+c+d=w}} 
\zeta\left(\;{\footnotesize\begin{ytableau}
\none & \none & d \\
a & b & c
\end{ytableau}\;}\right) \ne  S_w\left(\;{\footnotesize\begin{ytableau}
\none & \none & \, \\
\, & \, & \,
\end{ytableau}\;}\right) . \]
In the latter, the index $(d)$ is required to be admissible, i.e., $d\geq 2$. 
\end{remark}

Note that $S_w\begin{varray}s_1,&\ldots,&s_n\\ r_1,&\ldots,&r_n \end{varray}$ is nonzero 
only when 
\[w\ge s_1+\cdots+s_n+r_1+\cdots+r_n+n. \]

Our basic strategy of computing these sums on ribbons is 
to reduce the number of corners $n$ by using the following formula: 

\begin{prop}\label{prop:inductive}
Let $s_1,\ldots,s_n\ge 0$ and $r_1,\ldots,r_n>0$ be integers. 
For $1\le i\le n-1$ with $r_i\ge 2$, we have 
\begin{align}\label{eq:inductive}
S_w\begin{varray}s_1,&\ldots,&s_i,&s_{i+1},&\ldots,&s_n\\ 
r_1,&\ldots,&r_i,&r_{i+1},&\ldots,&r_n\end{varray}
+S_w\begin{varray}s_1,&\ldots,&s_i,&s_{i+1}+1,&\ldots,&s_n\\ 
r_1,&\ldots,&r_i-1,&r_{i+1},&\ldots,&r_n\end{varray}\\
=\sum_{\substack{w_1+w_2=w \\ w_1\ge s_1+\cdots+s_i+r_1+\cdots+r_i+i\\ w_2\ge s_{i+1}+\cdots+s_n+r_{i+1}+\cdots+r_n+n-i }}S_{w_1}\begin{varray}s_1,&\ldots,&s_i\\ r_1,&\ldots,&r_i \end{varray}\cdot 
S_{w_2}\begin{varray}s_{i+1},&\ldots,&s_n\\ r_{i+1},&\ldots,&r_n \end{varray}. 
\end{align}
\end{prop}
\begin{proof}
By switching the inequality $b_{i+1,1}<a_{i,1}$ in \eqref{eq:zeta(ll/kk)} (for the given $i$) 
to the opposite $a_{i,1}\le b_{i+1,1}$ and adding the corresponding Schur MZVs, we deduce that 
\begin{multline}
\zeta\begin{varray}
\bl_1,&\ldots,&\bl_i,&\bl_{i+1},&\ldots,&\bl_n\\
\bk_1,&\ldots,&\bk_i,&\bk_{i+1},&\ldots,&\bk_n
\end{varray}
+\zeta\begin{varray}
\bl_1,&\ldots,&\bl_i,&\bl_{i+1}',&\ldots,&\bl_n\\
\bk_1,&\ldots,&\bk_i',&\bk_{i+1},&\ldots,&\bk_n
\end{varray}\\
=\zeta\begin{varray}\bl_1,&\ldots,&\bl_i\\ \bk_1,&\ldots,&\bk_i \end{varray}\cdot 
\zeta\begin{varray}\bl_{i+1},&\ldots,&\bl_n\\ \bk_{i+1},&\ldots,&\bk_n \end{varray}, 
\end{multline}
where $\bk_i'=(k_{i,2},\ldots,k_{i,r_i})$ and $\bl_{i+1}'=(k_{i,1},l_{i+1,1},\ldots,l_{i+1,s_{i+1}})$. 
Then \eqref{eq:inductive} follows. 
\end{proof}

By using Proposition \ref{prop:inductive} repeatedly, 
the sums on general ribbons are expressed in terms of the values 
of the type $S_w\begin{varray}s,&0,&\ldots,&0\\ r_1,&r_2,&\ldots,&r_n\end{varray}$. 
For example, we have 
\begin{align}
S_w\begin{varray} s,&1\\ r_1,&r_2 \end{varray}
&=\sum_{\substack{w_1+w_2=w \\ w_1\ge s+r_1+2 \\w_2\ge r_2+1}}S_{w_1}\begin{varray} s\\ r_1+1\end{varray} S_{w_2}\begin{varray} 0\\ r_2\end{varray}
-S_w\begin{varray} s,&0\\ r_1+1,&r_2 \end{varray},\\
S_w\begin{varray} s,&2\\ r_1,&r_2 \end{varray}
&=\sum_{\substack{w_1+w_2=w \\ w_1\ge s+r_1+2 \\w_2\ge r_2+2}}S_{w_1}\begin{varray} s\\ r_1+1\end{varray} S_{w_2}\begin{varray} 1\\ r_2\end{varray}
-S_w\begin{varray} s,&1\\ r_1+1,&r_2 \end{varray}\\
&=\sum_{\substack{w_1+w_2=w \\ w_1\ge s+r_1+2 \\w_2\ge r_2+2}}S_{w_1}\begin{varray} s\\ r_1+1\end{varray} S_{w_2}\begin{varray} 1\\ r_2\end{varray}\\
&\qquad -\sum_{\substack{w_1+w_2=w \\ w_1\ge s+r_1+e \\w_2\ge r_2+1}}S_{w_1}\begin{varray} s\\ r_1+2\end{varray} 
S_{w_2}\begin{varray} 0\\ r_2\end{varray}+S_w\begin{varray} s,&0\\ r_1+2,&r_2 \end{varray}
\end{align}
and so on (cf.~Lemma \ref{lem:ribbon_2_corner_prelim}). 
For the latter type sums, the following formula holds: 

\begin{thm}\label{thm:s00}
For $w\ge 0$, $s\ge 0$ and $r_1,\ldots,r_n>0$, we have 
\begin{equation}\label{eq:s00}
S_w\begin{varray}
s, & 0, & \ldots, & 0\\
r_1, & r_2, & \ldots, & r_n
\end{varray}
=\sum_{\substack{t_1,\ldots,t_n\ge 0\\ t_1+\cdots+t_n=s}}
\sum_{\substack{w_i\ge r_i+t_i+1\\ w_1+\cdots+w_n=w}}
\prod_{i=1}^n\binom{w_i-1}{t_i}\zeta(w_1,\ldots,w_n). 
\end{equation}
\end{thm}
\begin{proof}
Put $r\coloneqq r_1+\cdots+r_n$. 
Then the left hand side is the sum of the series 
\begin{equation}\label{eq:linearize}
\ytableausetup{boxsize=1.2em}
\zeta\left(\;\begin{ytableau}
\none & \none & \none & k_1 \\
\none & \none & \none & \svdots \\
l_1 & \cdots & l_s & k_r
\end{ytableau}\;\right)=\sum_{0<a_1<\cdots<a_r\ge b_s\ge\cdots\ge b_1>0}
\frac{1}{a_1^{k_1}\cdots a_r^{k_r}b_1^{l_1}\cdots b_s^{l_s}},
\end{equation}
where $(l_1,\ldots,l_s)$ runs through indices of depth $s$ and $(k_1,\ldots,k_r)$ runs through 
indices of depth $r$ such that $k_u\ge 2$ for $u\in\{r_n,r_n+r_{n-1},\ldots,r\}$, 
satisfying $l_1+\cdots+l_s+k_1+\cdots+k_r=w$. 
By ``stuffling'', i.e., classifying all possible orders of $a_p$'s and $b_q$'s, 
this series is expanded into a certain sum of MZVs $\zeta(x_1,\ldots,x_{r+u})$ 
of weight $w$ with $0\le u\le s$. Here each of entries of $(x_1,\ldots,x_{r+u})$ is 
of the form 
\begin{equation}\label{eq:decompose}
k_p+l_q+\cdots+l_{q'} \text{ or } l_q+\cdots+l_{q'}, 
\end{equation}
where $l_q,\ldots,l_{q'}$ are some consecutive members of $l_1,\ldots,l_s$, 
possibly zero members in the former case and at least one member in the latter. 

Let us fix an index $(x_1,\ldots,x_{r+u})$ of weight $w$ with $0\le u\le s$ 
and count how many times $\zeta(x_1,\ldots,x_{r+u})$ appears in 
$S_w\begin{varray}
s, & 0, & \ldots, & 0\\
r_1, & r_2, & \ldots, & r_n
\end{varray}$
when all series \eqref{eq:linearize} are expanded as above. 
First, given $u_1,\ldots,u_n\ge 0$ satisfying $u_1+\cdots+u_n=u$, 
consider the cases that 
$x_{r_n+u_n+\cdots+r_i+u_i}$ contains $k_{r_n+\cdots+r_i}$ as a summand for $i=1,\ldots,n$. 
Then the number of possibilities of the places where other $k_p$'s appear is 
$\prod_{i=1}^n\binom{r_i-1+u_i}{u_i}$. Moreover, 
each entry of $(x_1,\ldots,x_{r+u})$ is decomposed into a sum of type \eqref{eq:decompose} 
and the total number of the plus symbol `$+$' is $s-u$. 
The number of possible such decompositions is 
\begin{align}
\binom{\sum_{i=1}^n\bigl\{(x_{h_{i+1}+1}-1)+\cdots+(x_{h_i-1}-1)+(x_{h_i}-2)\bigr\}}{s-u}&\\
=\binom{w-n-r-u}{s-u}&, 
\end{align}
where we set $h_i=r_n+u_n+\cdots+r_i+u_i$ and $h_{n+1}=0$. 
Therefore we obtain that 
\begin{align}
S_w\begin{varray}
s, & 0, & \ldots, & 0\\
r_1, & r_2, & \ldots, & r_n
\end{varray}
=\sum_{\substack{u_1,\ldots,u_n\ge 0\\ u\coloneqq u_1+\cdots+u_n\le s}}
\binom{w-n-r-u}{s-u}\prod_{i=1}^n\binom{r_i+u_i-1}{u_i}&\\
\times\sum_{\substack{w_1,\ldots,w_{r+u}\ge 1\\ w_p\ge 2\;(p=r_n+u_n+\cdots+r_i+u_i)\\ 
w_1+\cdots+w_{r+u}=w}}\zeta(w_1,\ldots,w_{r+u})&. 
\end{align}
By applying Ohno's relation to the last sum, we see that this is equal to 
\begin{align}\label{eq:Ohno}
S_w\begin{varray}
s, & 0, & \ldots, & 0\\
r_1, & r_2, & \ldots, & r_n
\end{varray}
=\sum_{\substack{u_1,\ldots,u_n\ge 0\\ u\coloneqq u_1+\cdots+u_n\le s}}
\binom{w-n-r-u}{s-u}\prod_{i=1}^n\binom{r_i+u_i-1}{u_i}&\\
\times\sum_{\substack{w_i\ge r_i+u_i+1\\ w_1+\cdots+w_n=w}}\zeta(w_1,\ldots,w_n)&. 
\end{align}
By means of the identity 
\[
\binom{w-n-r-u}{s-u}
=\sum_{\substack{v_1,\ldots,v_n\ge 0\\ v_1+\cdots+v_n=s-u}}
\prod_{i=1}^n\binom{w_i-1-r_i-u_i}{v_i}, 
\]
one can rewrite the expression \eqref{eq:Ohno} as 
\begin{align}
&\sum_{\substack{u_1,\ldots,u_n\ge 0\\ v_1,\ldots,v_n\ge 0\\ u_1+v_1+\cdots+u_n+v_n=s}}
\sum_{\substack{w_i\ge r_i+u_i+1\\ w_1+\cdots+w_n=w}}
\prod_{i=1}^n\binom{w_i-1-r_i-u_i}{v_i}\binom{r_i+u_i-1}{u_i}\zeta(w_1,\ldots,w_n)\\
&=\sum_{\substack{t_1,\ldots,t_n\ge 0\\ t_1+\cdots+t_n=s}}
\sum_{\substack{w_i\ge r_i+t_i+1\\ w_1+\cdots+w_n=w}}
\prod_{i=1}^n\Biggl(\sum_{u_i=0}^{t_i}\binom{w_i-1-r_i-u_i}{t_i-u_i}\binom{r_i+u_i-1}{u_i}\Biggr)
\zeta(w_1,\ldots,w_n). 
\end{align}
Here we put $t_i=u_i+v_i$ and use that 
\[\binom{w_i-1-r_i-u_i}{t_i-u_i}=0\quad \text{ when $t_i>u_i$ and } r_i+u_i+1\le w_i\le r_i+t_i. \]
Thus the theorem follows from the identity 
\begin{align}
\sum_{u_i=0}^{t_i}\binom{w_i-1-r_i-u_i}{t_i-u_i}\binom{r_i+u_i-1}{u_i}
&=\binom{w_i-1}{t_i}. \qedhere
\end{align}
\end{proof}

\begin{cor}\label{cor:ncornerribbon}
For $n\geq 1$ and integers $s_1,\ldots,s_n\ge 0$ and $r_1,\ldots,r_n>0$ the sum 
$S_w\binom{s_1,\ldots,s_n}{r_1,\ldots,r_n}$ 
can be written as a $\Q$-linear combination of MZVs of weight $w$ and depth $\leq n$.
\end{cor}
\begin{proof}
This is a direct consequence of Proposition \ref{prop:inductive} and Theorem \ref{thm:s00}.
\end{proof}
\begin{cor}\label{cor:symmetric sum}
For any $s\ge 0$ and $r_1,\ldots,r_n>0$, the symmetric sum 
\[\sum_{\sigma\in\mathfrak{S}_n} 
S_w\begin{varray} s,&0,&\ldots,&0\\ r_{\sigma(1)},&r_{\sigma(2)},&\ldots,&r_{\sigma(n)}\end{varray}\]
is a polynomial in single zeta values. In particular, 
$S_w\begin{varray} s,&0,&\ldots,&0\\ r,&r,&\ldots,&r\end{varray}$ 
is a polynomial in single zeta values. 
\end{cor}
\begin{proof}
By \eqref{eq:s00}, this symmetric sum is equal to 
\[\sum_{\substack{t_1,\ldots,t_n\ge 0\\ t_1+\cdots+t_n=s}}
\sum_{\substack{w_i\ge r_i+t_i+1\\ w_1+\cdots+w_n=w}}
\prod_{i=1}^n\binom{w_i-1}{t_i}\sum_{\sigma\in\mathfrak{S}_n} \zeta(w_{\sigma(1)},\ldots,w_{\sigma(n)}). \]
Thus the claim follows from Hoffman's symmetric sum formula \cite[Theorem 2.2]{Hoffman92}. 
\end{proof}

\begin{example}
The case of $n=1$ will be treated in Theorem \ref{thm:anti-hook}. 
Here let us consider the case of $n=2$ and $r_1=r_2=r$. 
We have 
\begin{align}
S_w\begin{varray} s,& 0\\ r,& r \end{varray}
&=\sum_{\substack{t_1,t_2\ge 0\\ t_1+t_2=s}}\sum_{\substack{w_i\ge r+t_i+1\\ w_1+w_2=w}}
\binom{w_1-1}{t_1}\binom{w_2-1}{t_2}\zeta(w_1,w_2)\\
&=\frac{1}{2}\sum_{\substack{t_1,t_2\ge 0\\ t_1+t_2=s}}\sum_{\substack{w_i\ge r+t_i+1\\ w_1+w_2=w}}
\binom{w_1-1}{t_1}\binom{w_2-1}{t_2}\bigl(\zeta(w_1,w_2)+\zeta(w_2,w_1)\bigr)\\
&=\frac{1}{2}\sum_{\substack{t_1,t_2\ge 0\\ t_1+t_2=s}}\sum_{\substack{w_i\ge r+t_i+1\\ w_1+w_2=w}}
\binom{w_1-1}{t_1}\binom{w_2-1}{t_2}\bigl(\zeta(w_1)\zeta(w_2)-\zeta(w)\bigr). 
\end{align}
This is a kind of sum formula of polynomial type. 
For a sum formula of bounded type, see \S\ref{sec:general_nagoya} (for instance see Example~\ref{ex:hook case} for a hook case). 
\end{example}

\subsection{Sum formulas of single type}
\label{subsec:single_type}

In this subsection,
we present sum formulas of single type for two special types of ribbons. 
The first is a simultaneous generalization of the classical sum formulas 
for the MZVs and MZSVs stated in \eqref{eq:mzvsumformula}.

\begin{thm}\label{thm:anti-hook}
For any integers $r\ge 1$, $s\ge 0$ and $w\ge s+r+1$, we have 
\begin{equation}\label{eq:anti-hook}
S_w\begin{varray}s \\ r\end{varray}=\binom{w-1}{s}\zeta(w). 
\end{equation}
\end{thm}
\begin{proof}
This immediately follows from Theorem \ref{thm:s00}.
\end{proof}

The second formula is for the ``stair of tread one'' shape.
The proof is an application of Proposition \ref{prop:inductive} 
and Theorem \ref{thm:s00}. 

\begin{thm}\label{thm:stair 1}
For any integers $r\ge 1$, $n\ge 1$ and $w\ge (r+2)n+1$, we have 
\begin{equation}\label{eq:stair 1}
S_w\begin{varray}\{1\}^{n-1},&1\\ \{r\}^{n-1},&r+1\end{varray}
=c_{w,r}(n)\zeta(w), 
\end{equation}
where
\[
c_{w,r}(n)\coloneqq \frac{w-1}{n}\binom{w-(r+1)n-2}{n-1}. 
\]
\end{thm}
See \eqref{eq:exampletreadone} in the introduction for examples in the cases $r=n=2$ and $r=1$, $n=3$.
\begin{remark}
The coefficient $c_{w,r}(n)$ is a positive integer. In fact, 
\[
c_{w,r}(n)=(r+1)\binom{w-(r+1)n-2}{n-1}+\binom{w-(r+1)n-1}{n}. 
\]
\end{remark}

\begin{proof}[Proof of Theorem \ref{thm:stair 1}]
We prove \eqref{eq:stair 1} by induction on $n$. 
For $n=1$, this is a special case of \eqref{eq:anti-hook}. 
Let us assume $n>1$ and for $1\le i\le n$, put 
\[
S_{w,r}(n,i)\coloneqq S_w\begin{varray}
\{1\}^{i-1}, & 1,   &\{0\}^{n-i}\\
\{r\}^{i-1}, & r+1, &\{r+1\}^{n-i}
\end{varray}, \qquad 
T_{w,r}(n)\coloneqq S_w\begin{varray}
\{0\}^n\\
\{r+1\}^n
\end{varray}. 
\]
Then, for $1\le i\le n-1$, Proposition \ref{prop:inductive} shows that 
\[
S_{w,r}(n,i)+S_{w,r}(n,i+1)
=\sum_{\substack{w_1\ge (r+2)i+1\\ w_2\ge (r+2)(n-i)}}S_{w_1,r}(i,i)\cdot T_{w_2,r}(n-i)
\]
(here and in what follows, we omit the ``total weight $=w$'' condition like $w_1+w_2=w$). 
The induction hypothesis gives $S_{w_1,r}(i,i)=c_{w_1,r}(i)\zeta(w_1)$ for $1\le i\le n-1$, 
while Theorem \ref{thm:s00} shows that 
\[
T_{w_2,r}(n-i)=\sum_{\substack{w'_1,\ldots,w'_{n-i}\ge r+2\\ w'_1+\cdots+w'_{n-i}=w_2}}
\zeta(w'_1,\ldots,w'_{n-i}). 
\]
Hence we obtain 
\begin{equation}\label{eq:A_i+B_i}
\begin{split}
S_{w,r}(n,i)+S_{w,r}(n,i+1)
&=\sum_{\substack{w_1\ge (r+2)i+1\\ w'_1,\ldots,w'_{n-i}\ge r+2}}
c_{w_1,r}(i)\zeta(w_1)\zeta(w'_1,\ldots,w'_{n-i})\\
&=A_i+B_i, 
\end{split}
\end{equation}
where $A_i$ and $B_i$ are defined by 
\begin{align}
A_i&\coloneqq\sum_{\substack{w_1\ge (r+2)i+1\\ w'_1,\ldots,w'_{n-i}\ge r+2}}
c_{w_1,r}(i)\bigl\{\zeta(w_1,w'_1,\ldots,w'_{n-i})+\cdots+\zeta(w'_1,\ldots,w'_{n-i},w_1)\bigr\}\notag\\
&=\sum_{j=1}^{n-i+1}\sum_{\substack{w_1,\ldots,w_{n-i+1}\ge r+2\\ w_j\ge (r+2)i+1}}
c_{w_j,r}(i)\zeta(w_1,\ldots,w_{n-i+1})\label{eq:A_i}
\end{align}
and 
\begin{align}
B_i&\coloneqq\sum_{\substack{w_1\ge (r+2)i+1\\ w'_1,\ldots,w'_{n-i}\ge r+2}}
c_{w_1,r}(i)\bigl\{\zeta(w_1+w'_1,\ldots,w'_{n-i})+\cdots+\zeta(w'_1,\ldots,w_1+w'_{n-i})\bigr\}\\
&=\sum_{j=1}^{n-i}\sum_{\substack{w_1,\ldots,w_{n-i}\ge r+2\\ w_j\ge (r+2)(i+1)+1}}
\sum_{a=(r+2)i+1}^{w_j-(r+2)} c_{a,r}(i)\zeta(w_1,\ldots,w_{n-i}). 
\end{align}
Note that the definition of $A_i$ works for $i=n$ and gives $A_n=c_{w,r}(n)\zeta(w)$. 

For $1\le i\le n-1$, we have $A_{i+1}=B_i$ since 
\begin{align}
\sum_{a=(r+2)i+1}^{w_j-(r+2)}c_{a,r}(i)
&=\sum_{b=0}^k\frac{(r+2)i+b}{i}\binom{b+i-1}{i-1}\qquad (k\coloneqq w_j-(r+2)(i+1)-1)\\
&=(r+1)\sum_{b=0}^k\binom{b+i-1}{i-1}+\sum_{b=0}^k\binom{b+i}{i}\\
&=(r+1)\binom{k+i}{i}+\binom{k+i+1}{i+1}=c_{w_j,r}(i+1). 
\end{align}
Thus \eqref{eq:A_i+B_i} says that 
\[S_{w,r}(n,i)+S_{w,r}(n,i+1)=A_i+A_{i+1}. \]
This implies inductively that $S_{w,r}(n,i)=A_i$, starting from 
\[S_{w,r}(n,1)=\sum_{j=1}^n\sum_{\substack{w_1,\ldots,w_n\ge r+2\\ w_j\ge r+3}}
(w_j-1)\zeta(w_1,\ldots,w_n)=A_1, \] 
which is a consequence of \eqref{eq:s00} and \eqref{eq:A_i}. 
The final identity $S_{w,r}(n,n)=A_n=c_{w,r}(n)\zeta(w)$ is exactly the desired formula. 
\end{proof}

\subsection{Two corners}
\label{sec:general_nagoya}
We next study ribbon shapes with two corners which give a bounded type sum formula.
The key ingredient is the sum formula weighted by the binomial coefficients
(\cref{prop:bwsumformula}).
For ribbons with two corners, we only use the case $d=2$.
For ease of calculation,
we rewrite \cref{cor:pldep2} in the following form
with $m_1=k_1-1$ and $m_2=k_2-1$:
For integers $m_1,m_2\ge0$ and $w\ge m_1+m_2+3$,
we have
\begin{equation}
\label{eq:weighted_sum_ready_to_apply}
\begin{aligned}
&\sum_{\substack{w_1\ge 1,w_2\ge 2\\w_1+w_2=w}}
\binom{w_1-1}{m_1}\binom{w_2-1}{m_2}
\zeta(w_1,w_2)\\
&=
(-1)^{m_2+1}
\sum_{\substack{w_1,w_2\ge2\\w_1+w_2=w}}
(-1)^{w_1}
\binom{w_1-1}{m_2}\binom{w_2-1}{m_1+m_2+1-w_1}
\zeta(w_1)\zeta(w_2)\\
&\hspace{20mm}
+(-1)^{m_1+1}
\sum_{\substack{w_1\ge 1, w_2\ge2\\ w_1+w_2=w}}
\binom{w_1-1}{m_1}\binom{w_2-1}{m_1+m_2+1-w_1}\zeta(w_1,w_2)\\
&\hspace{30mm}
+\mathbbm{1}_{m_2=0}\binom{w-2}{m_1}\bigl(\zeta(w)+\zeta(1,w-1)\bigr).
\end{aligned}
\end{equation}
We then start with a preliminary calculation:
\begin{lemma}
\label{lem:ribbon_2_corner_prelim}
For $s_1,s_2\ge0$, $r_1,r_2>0$ and $w\ge s_1+s_2+r_1+r_2+2$, we have
\begin{align}
S_w\begin{varray}
s_1, & s_2\\
r_1, & r_2
\end{varray}
&=
\sum_{i=0}^{s_2-1}
(-1)^{s_2-i-1}
\sum_{\substack{w_1\ge s_1+s_2+r_1-i+1\\w_2\ge r_2+i+1\\w_1+w_2=w}}
\binom{w_1-1}{s_1}\binom{w_2-1}{i}\zeta(w_1)\zeta(w_2)\\
&\hspace{20mm}
+(-1)^{s_2}
\sum_{\substack{t_1,t_2\ge0\\t_1+t_2=s_1}}
\sum_{\substack{w_1\ge s_2+r_1+t_1+1\\w_2\ge r_2+t_2+1\\w_1+w_2=w}}
\binom{w_1-1}{t_1}\binom{w_2-1}{t_2}\zeta(w_1,w_2).
\end{align}
\end{lemma}
\begin{proof}
By a repeated application of Proposition~\ref{prop:inductive},
we obtain
\begin{align}
S_{w}
\begin{varray}
s_{1}, & s_{2}\\
r_{1}, & r_{2}
\end{varray}
&=
\sum_{w_1+w_2=w}
S_{w_1}
\begin{varray}
s_{1}\\
r_{1}+1
\end{varray}
S_{w_2}
\begin{varray}
s_{2}-1\\
r_{2}
\end{varray}
-
S_{w}\begin{varray}
s_{1},  & s_{2}-1\\
r_{1}+1, & r_{2}
\end{varray}\\
&=\cdots\\
&=
\sum_{i=1}^{s_2}
(-1)^{i-1}
\sum_{w_1+w_2=w}
S_{w_1}
\begin{varray}
s_{1}\\
r_{1}+i
\end{varray}
S_{w_2}
\begin{varray}
s_{2}-i\\
r_{2}
\end{varray}
+(-1)^{s_2}
S_{w}\begin{varray}
s_{1},     & 0\\
s_2+r_{1}, & r_{2}
\end{varray}\\
&=
\sum_{i=0}^{s_2-1}
(-1)^{s_2-i-1}
\sum_{w_1+w_2=w}
S_{w_1}
\begin{varray}
s_{1}\\
s_2+r_{1}-i
\end{varray}
S_{w_2}
\begin{varray}
i\\
r_{2}
\end{varray}
+(-1)^{s_2}
S_{w}\begin{varray}
s_{1},     & 0\\
s_2+r_{1}, & r_{2}
\end{varray}.
\end{align}
By applying Theorem~\ref{thm:anti-hook}
to the first sum with taking care of admissible range
and by applying Theorem~\ref{thm:s00}
to the second sum, we obtain the lemma.
\end{proof}


As an immediate consequence of the expression above, 
one finds that there is indeed a polynomial type sum formula
for a specific class of ribbons with two corners, that is, the case $r_2=s_2+r_1$.

\begin{thm}
\label{thm:Yamasaki_polynomial_observation}
For $s_1,s_2\ge0$, $r_1>0$ and $w\ge s_1+2s_2+2r_1+2$, we have
\[
S_w\begin{varray}
s_1, & s_2\\
r_1, & s_2+r_1
\end{varray}
\]
is a polynomial in single zeta values.
\end{thm}
\begin{proof}
By Lemma~\ref{lem:ribbon_2_corner_prelim} with $r_2=s_2+r_1$,
it suffices to show
\[
\sum_{\substack{t_1,t_2\ge0\\t_1+t_2=s_1}}
\sum_{\substack{w_1\ge s_2+r_1+t_1+1\\w_2\ge s_2+r_1+t_2+1\\w_1+w_2=w}}
\binom{w_1-1}{t_1}\binom{w_2-1}{t_2}\zeta(w_1,w_2)
\]
can be written as a polynomial of single zeta values.
By symmetry, this sum is
\[
=
\frac{1}{2}
\sum_{\substack{t_1,t_2\ge0\\t_1+t_2=s_1}}
\sum_{\substack{w_1\ge s_2+r_1+t_1+1\\w_2\ge s_2+r_1+t_2+1\\w_1+w_2=w}}
\binom{w_1-1}{t_1}\binom{w_2-1}{t_2}
\bigl(\zeta(w_1,w_2)+\zeta(w_2,w_1)\bigr).
\]
Then, the result follows by the harmonic product formula. 
\end{proof}

The next theorem is a sum formula for general ribbons with two corners, which is of bounded type.
Although the explicit formula itself is rather complicated,
it is a direct consequence of \cref{eq:weighted_sum_ready_to_apply} and \cref{lem:ribbon_2_corner_prelim}.

\begin{thm}
\label{thm:ribbon_2_corner}
For $s_1,s_2\ge0$, $r_1,r_2>0$ and $w\ge s_1+s_2+r_1+r_2+2$, we have 
\begin{equation}
\label{thm:ribbon_2_corner:result}
\begin{aligned}
S_w\begin{varray}
s_1, & s_2\\
r_1, & r_2
\end{varray}
&=\binom{w-2}{s_1+s_2}\zeta(w)\\
&\hspace{10mm}
+\sum_{\substack{w_1, w_2\ge2\\w_1+w_2=w}}
A_{w_1,w_2}^{s_1,s_2,r_1,r_2}\zeta(w_1)\zeta(w_2)
+
\sum_{\substack{w_1\ge1, w_2\ge2\\w_1+w_2=w}}
B_{w_1,w_2}^{s_1,s_2,r_1,r_2}\zeta(w_1,w_2),
\end{aligned}
\end{equation}
where the integers
$A_{w_1,w_2}^{s_1,s_2,r_1,r_2}$
and
$B_{w_1,w_2}^{s_1,s_2,r_1,r_2}$
are given by
\begin{align}
A_{w_1,w_2}^{s_1,s_2,r_1,r_2}
&\coloneqq
(-1)^{w_1}C_{w_1,w_2}^{s_1,s_2}\\
&\hspace{10mm}
-\mathbbm{1}_{w_1\le s_1+r_1\ \textup{or}\ w_2\le s_2+r_2-1}
\binom{w_1-1}{s_1}\binom{w_2-2}{s_2-1}\\
&\hspace{10mm}
+\mathbbm{1}_{w_1>s_1+r_1}(-1)^{s_1+r_1+w_1}
\binom{w_1-1}{s_1}
\binom{w_2-2}{s_1+s_2+r_1-w_1}\\
&\hspace{10mm}
+\mathbbm{1}_{r_2<w_2\le s_2+r_2-1}
(-1)^{s_2+r_2+w_2}
\binom{w_1-1}{s_1}\binom{w_2-2}{r_2-1},
\\
B_{w_1,w_2}^{s_1,s_2,r_1,r_2}
&\coloneqq
C_{w_1,w_2}^{s_1,s_2}
-
(-1)^{s_2}
\sum_{\substack{
t_1,t_2\ge0\\
t_1\ge w_1-(s_2+r_1)\ \textup{or}\ t_2\ge w_2-r_2\\
t_1+t_2=s_1}}
\binom{w_1-1}{t_1}\binom{w_2-1}{t_2}
\end{align}
with
\[
C_{w_1,w_2}^{s_1,s_2}
\coloneqq
(-1)^{s_2}
\sum_{\substack{0\le i\le s_1\\1\le j\le s_2\\i+j=w_1}}
\binom{w_1-1}{i}\binom{w_2-1}{s_1-i}
-(-1)^{s_1}
\binom{w_1-1}{s_1}
\binom{w_2-2}{s_1+s_2-w_1}.
\]
\end{thm}
\begin{remark}
By our convention on binomial coefficients,
$A_{w_1,w_2}^{s_1,s_2,r_1,r_2}=B_{w_1,w_2}^{s_1,s_2,r_1,r_2}=0$ unless 
\[
w_1\le s_1+s_2+r_1 \quad \text{ or } \quad  w_2\le \max(s_2+r_2-1,s_1+r_2)
\]
and so \cref{thm:ribbon_2_corner:result}
is a bounded type sum formula
after expanding the product $\zeta(w_1)\zeta(w_2)$ by the harmonic product formula.
\end{remark}
\begin{proof}[Proof of \cref{thm:ribbon_2_corner}]
Write the equation in Lemma~\ref{lem:ribbon_2_corner_prelim} as
\[
S_w\begin{varray}
s_1, & s_2\\
r_1, & r_2
\end{varray}
=
S_1+S_2
\]
with
\begin{align}
S_1
&\coloneqq\sum_{i=0}^{s_2-1}
(-1)^{s_2-i-1}
\sum_{\substack{w_1\ge s_1+s_2+r_1-i+1\\w_2\ge r_2+i+1\\w_1+w_2=w}}
\binom{w_1-1}{s_1}\binom{w_2-1}{i}\zeta(w_1)\zeta(w_2),\\
S_2
&\coloneqq(-1)^{s_2}
\sum_{\substack{t_1,t_2\ge0\\t_1+t_2=s_1}}
\sum_{\substack{w_1\ge s_2+r_1+t_1+1\\w_2\ge r_2+t_2+1\\w_1+w_2=w}}
\binom{w_1-1}{t_1}\binom{w_2-1}{t_2}\zeta(w_1,w_2).
\end{align}


The sum $S_1$ can be rewritten as $S_1=S_{11}-S_{12}-S_{13}$ where  
\begin{align}
S_{11}
&\coloneqq
\sum_{i=0}^{s_2-1}
(-1)^{s_2-i-1}
\sum_{\substack{w_1,w_2\ge2\\w_1+w_2=w}}
\binom{w_1-1}{s_1}\binom{w_2-1}{i}\zeta(w_1)\zeta(w_2)\\
S_{12}&\coloneqq
\sum_{i=0}^{s_2-1}
(-1)^{s_2-i-1}
\sum_{\substack{2\le w_1\le s_1+s_2+r_1-i\\w_1+w_2=w}}
\binom{w_1-1}{s_1}\binom{w_2-1}{i}\zeta(w_1)\zeta(w_2)\\
S_{13}&\coloneqq
\sum_{i=0}^{s_2-1}
(-1)^{s_2-i-1}
\sum_{\substack{2\le w_2\le r_2+i\\w_1+w_2=w}}
\binom{w_1-1}{s_1}\binom{w_2-1}{i}\zeta(w_1)\zeta(w_2).
\end{align}
By the harmonic product formula, we have
$S_{11}=S_{111}+S_{112}+S_{113}$
with
\begin{align}
S_{111}
&\coloneqq
\sum_{i=0}^{s_2-1}
(-1)^{s_2-i-1}
\sum_{\substack{w_1,w_2\ge 2\\w_1+w_2=w}}
\binom{w_1-1}{s_1}\binom{w_2-1}{i}\zeta(w_1,w_2),\\
S_{112}
&\coloneqq
\sum_{i=0}^{s_2-1}
(-1)^{s_2-i-1}
\sum_{\substack{w_1,w_2\ge 2\\w_1+w_2=w}}
\binom{w_1-1}{i}\binom{w_2-1}{s_1}\zeta(w_1,w_2),\\
S_{113}
&\coloneqq
\biggl\{
\sum_{i=0}^{s_2-1}
(-1)^{s_2-i-1}
\sum_{\substack{w_1,w_2\ge 2\\w_1+w_2=w}}
\binom{w_1-1}{s_1}\binom{w_2-1}{i}
\biggr\}
\zeta(w).
\end{align}
For the sum $S_{111}$,
after supplementing the term with $w_1=1$,
by \cref{eq:weighted_sum_ready_to_apply}, we get
\begin{align}
S_{111}
&=
(-1)^{s_2}
\sum_{i=0}^{s_2-1}
\sum_{\substack{w_1,w_2\ge2\\w_1+w_2=w}}
(-1)^{w_1}
\binom{w_1-1}{i}\binom{w_2-1}{s_1+i+1-w_1}\zeta(w_1)\zeta(w_2)\\
&\hspace{5mm}
+
\sum_{i=0}^{s_2-1}(-1)^{s_1+s_2-i}
\sum_{\substack{w_1\ge 1, w_2\ge2\\ w_1+w_2=w}}
\binom{w_1-1}{s_1}\binom{w_2-1}{s_1+i+1-w_1}\zeta(w_1,w_2)\\
&\hspace{10mm}
+\mathbbm{1}_{s_2>0}(-1)^{s_2-1}\binom{w-2}{s_1}
\bigl(\zeta(w)+\zeta(1,w-1)\bigr)\\
&\hspace{15mm}
-\mathbbm{1}_{s_1=0}
\sum_{i=0}^{s_2-1}(-1)^{s_2-i-1}\binom{w-2}{i}\zeta(1,w-1),
\end{align}
where we should note that the sum $S_{111}$ is empty if $s_2=0$ and $\mathbbm{1}_{s_2>0}$
is inserted to cover such a degenerate case.
By swapping the summation
and changing variable via $i\leadsto w_1-i-1$, we have
\begin{align}
&(-1)^{s_2}
\sum_{i=0}^{s_2-1}
\sum_{\substack{w_1,w_2\ge2\\w_1+w_2=w}}
(-1)^{w_1}
\binom{w_1-1}{i}\binom{w_2-1}{s_1+i+1-w_1}\zeta(w_1)\zeta(w_2)\\
&=
(-1)^{s_2}
\sum_{\substack{w_1,w_2\ge2\\w_1+w_2=w}}
(-1)^{w_1}\zeta(w_1)\zeta(w_2)
\sum_{\substack{0\le i\le s_1\\1\le j\le s_2\\i+j=w_1}}
\binom{w_1-1}{i}\binom{w_2-1}{s_1-i}.
\end{align}
Since $\binom{w_2-1}{s_1+i+1-w_1}=0$ if $i<w_1-s_1-1$,
by changing variable via
$i\leadsto i+w_1-s_1-1$, we have
\begin{align}
&\sum_{i=0}^{s_2-1}(-1)^{s_1+s_2-i}
\sum_{\substack{w_1\ge 1, w_2\ge2\\ w_1+w_2=w}}
\binom{w_1-1}{s_1}\binom{w_2-1}{s_1+i+1-w_1}\zeta(w_1,w_2)\\
&=
-
(-1)^{s_1}
\sum_{\substack{w_1\ge1,w_2\ge2\\w_1+w_2=w}}
\binom{w_1-1}{s_1}\binom{w_2-2}{s_1+s_2-w_1}
\zeta(w_1,w_2).
\end{align}
Combining the above two formulas with
$\sum_{i=0}^{s_2-1}(-1)^{s_2-i-1}\binom{w-2}{i}
=\binom{w-3}{s_2-1}$,
we have
\begin{align}
S_{111}
&=
(-1)^{s_2}
\sum_{\substack{w_1,w_2\ge2\\w_1+w_2=w}}
(-1)^{w_1}\zeta(w_1)\zeta(w_2)
\sum_{\substack{0\le j\le s_1\\1\le i\le s_2\\i+j=w_1}}
\binom{w_1-1}{j}\binom{w_2-1}{s_1-j}\\
&\hspace{5mm}
-(-1)^{s_1}
\sum_{\substack{w_1\ge1,w_2\ge2\\w_1+w_2=w}}
\binom{w_1-1}{s_1}\binom{w_2-2}{s_1+s_2-w_1}
\zeta(w_1,w_2)\\
&\hspace{10mm}
-\mathbbm{1}_{s_2>0}(-1)^{s_2}
\binom{w-2}{s_1}\bigl(\zeta(w)+\zeta(1,w-1)\bigr)
-\mathbbm{1}_{s_1=0}
\binom{w-3}{s_2-1}\zeta(1,w-1).
\end{align}
Similarly, for the sum $S_{112}$,
we can apply \cref{eq:weighted_sum_ready_to_apply}
to get
\begin{align}
S_{112}
&=
\sum_{i=0}^{s_2-1}(-1)^{s_1+s_2-i}
\sum_{\substack{w_1,w_2\ge2\\w_1+w_2=w}}
(-1)^{w_1}
\binom{w_1-1}{s_1}\binom{w_2-1}{s_1+i+1-w_1}
\zeta(w_1)\zeta(w_2)\\
&\hspace{10mm}
+(-1)^{s_2}
\sum_{i=0}^{s_2-1}
\sum_{\substack{w_1\ge 1, w_2\ge2\\ w_1+w_2=w}}
\binom{w_1-1}{i}\binom{w_2-1}{s_1+i+1-w_1}\zeta(w_1,w_2)\\
&\hspace{30mm}
+\mathbbm{1}_{s_1=0}\sum_{i=0}^{s_2-1}(-1)^{s_2-i-1}\binom{w-2}{i}
\bigl(\zeta(w)+\zeta(1,w-1)\bigr)\\
&\hspace{40mm}
+\mathbbm{1}_{s_2>0}
(-1)^{s_2}\binom{w-2}{s_1}\zeta(1,w-1).
\end{align}
By a calculation similar to that of $S_{111}$, we have
\begin{align}
S_{112}
&=
-(-1)^{s_1}
\sum_{\substack{w_1,w_2\ge2\\w_1+w_2=w}}
(-1)^{w_1}
\binom{w_1-1}{s_1}\binom{w_2-2}{s_1+s_2-w_1}
\zeta(w_1)\zeta(w_2)\\
&\hspace{5mm}
+(-1)^{s_2}
\sum_{\substack{w_1\ge1,w_2\ge2\\w_1+w_2=w}}
\zeta(w_1,w_2)
\sum_{\substack{0\le i\le s_1\\1\le j\le s_2\\i+j=w_1}}
\binom{w_1-1}{i}\binom{w_2-1}{s_1-i}\\
&\hspace{15mm}
+\mathbbm{1}_{s_1=0}\binom{w-3}{s_2-1}\bigl(\zeta(w)+\zeta(1,w-1)\bigr)
+\mathbbm{1}_{s_2>0}(-1)^{s_2}\binom{w-2}{s_1}\zeta(1,w-1).
\end{align}
For the sum $S_{113}$,
by $\sum_{i=0}^{s_2-1}(-1)^{s_2-i-1}\binom{w_2-1}{i}=\binom{w_2-2}{s_2-1}$ we have
\begin{align}
S_{113}
&=\biggl\{\sum_{i=0}^{s_2-1}
(-1)^{s_2-i-1}
\sum_{\substack{w_1,w_2\ge2\\w_1+w_2=w}}
\binom{w_1-1}{s_1}\binom{w_2-1}{i}\biggr\}\zeta(w)\\
&=
\biggl\{
\sum_{\substack{w_1\ge1, w_2\ge2\\w_1+w_2=w}}
\binom{w_1-1}{s_1}
\binom{w_2-2}{s_2-1}
-\mathbbm{1}_{s_1=0}\binom{w-3}{s_2-1}
\biggr\}\zeta(w)\\
&=
\biggl\{
\mathbbm{1}_{s_2>0}\binom{w-2}{s_1+s_2}
-\mathbbm{1}_{s_1=0}\binom{w-3}{s_2-1}
\biggr\}\zeta(w).
\end{align}

We next consider the sum $S_{12}$ and $S_{13}$.
By swapping the summation, we have
\begin{align}
S_{12}
&=
\sum_{\substack{2\le w_1\le s_1+s_2+r_1\\w_1+w_2=w}}
\binom{w_1-1}{s_1}\zeta(w_1)\zeta(w_2)
\sum_{i=0}^{\min(s_2-1,s_1+s_2+r_1-w_1)}
(-1)^{s_2-i-1}\binom{w_2-1}{i}\\
&=
\sum_{\substack{2\le w_1\le s_1+r_1\\w_1+w_2=w}}
\binom{w_1-1}{s_1}\zeta(w_1)\zeta(w_2)
\sum_{i=0}^{s_2-1}
(-1)^{s_2-i-1}\binom{w_2-1}{i}\\
&\hspace{10mm}
+
\sum_{\substack{s_1+r_1<w_1\le s_1+s_2+r_1\\w_1+w_2=w}}
\binom{w_1-1}{s_1}\zeta(w_1)\zeta(w_2)
\sum_{i=0}^{s_1+s_2+r_1-w_1}
(-1)^{s_2-i-1}\binom{w_2-1}{i}\\
&=
\sum_{\substack{2\le w_1\le s_1+r_1\\w_1+w_2=w}}
\binom{w_1-1}{s_1}\binom{w_2-2}{s_2-1}
\zeta(w_1)\zeta(w_2)\\
&\hspace{10mm}
-
\sum_{\substack{s_1+r_1<w_1\le s_1+s_2+r_1\\w_1+w_2=w}}
(-1)^{s_1+r_1+w_1}
\binom{w_1-1}{s_1}
\binom{w_2-2}{s_1+s_2+r_1-w_1}
\zeta(w_1)\zeta(w_2).
\end{align}
Similarly, we have
\begin{align}
S_{13}
&=
\sum_{\substack{2\le w_2\le s_2+r_2-1\\w_1+w_2=w}}
\binom{w_1-1}{s_1}\binom{w_2-2}{s_2-1}\zeta(w_1)\zeta(w_2)\\
&\hspace{10mm}
-\sum_{\substack{r_2<w_2\le s_2+r_2-1\\w_1+w_2=w}}
(-1)^{s_2+r_2+w_2}
\binom{w_1-1}{s_1}\binom{w_2-2}{r_2-1}\zeta(w_1)\zeta(w_2).
\end{align}

For the sum $S_{2}$, we can dissect as
$S_{2}=S_{21}-S_{22}-S_{23}$
by writing
\begin{align}
S_{21}
&\coloneqq
(-1)^{s_2}
\sum_{\substack{t_1,t_2\ge0\\t_1+t_2=s_1}}
\sum_{\substack{w_1\ge 1,w_2\ge 2\\w_1+w_2=w}}
\binom{w_1-1}{t_1}\binom{w_2-1}{t_2}\zeta(w_1,w_2),\\
S_{22}
&\coloneqq
(-1)^{s_2}
\sum_{\substack{t_1,t_2\ge0\\t_1+t_2=s_1}}
\sum_{\substack{1\le w_1\le s_2+r_1+t_1\\w_1+w_2=w}}
\binom{w_1-1}{t_1}\binom{w_2-1}{t_2}\zeta(w_1,w_2),\\
S_{23}
&\coloneqq
(-1)^{s_2}
\sum_{\substack{t_1,t_2\ge0\\t_1+t_2=s_1}}
\sum_{\substack{2\le w_2\le r_2+t_2\\w_1+w_2=w}}
\binom{w_1-1}{t_1}\binom{w_2-1}{t_2}\zeta(w_1,w_2).
\end{align}
For $S_{21}$, we can calculate the sum over binomial coefficients and get
\begin{align}
S_{21}
=
(-1)^{s_2}\binom{w-2}{s_1}
\sum_{\substack{w_1\ge 1,w_2\ge 2\\w_1+w_2=w}}
\zeta(w_1,w_2)
=
(-1)^{s_2}\binom{w-2}{s_1}\zeta(w)
\end{align}
by the usual sum formula.
For $S_{22},S_{23}$, note that
\begin{align}
S_{22}
&=
(-1)^{s_2}
\sum_{\substack{w_1\ge1,w_2\ge2\\w_1+w_2=w}}
\zeta(w_1,w_2)
\sum_{\substack{t_1,t_2\ge0\\t_1\ge w_1-(s_2+r_1)\\t_1+t_2=s_1}}
\binom{w_1-1}{t_1}\binom{w_2-1}{t_2},
\\
S_{23}
&=
(-1)^{s_2}
\sum_{\substack{w_1\ge1,w_2\ge2\\w_1+w_2=w}}
\zeta(w_1,w_2)
\sum_{\substack{t_1,t_2\ge0\\t_2\ge w_2-r_2\\t_1+t_2=s_1}}
\binom{w_1-1}{t_1}\binom{w_2-1}{t_2}.
\end{align}
Since the range of $(t_1,t_2)$ for $S_{22}$ and $S_{23}$ are disjoint, we have
\begin{align}
S_{22}
+
S_{23}
=
(-1)^{s_2}
\sum_{\substack{w_1\ge1,w_2\ge2\\w_1+w_2=w}}
\zeta(w_1,w_2)
\sum_{\substack{t_1,t_2\ge0\\
t_1\ge w_1-(s_2+r_1)\ \text{or}\ t_2\ge w_2-r_2\\
t_1+t_2=s_1}}
\binom{w_1-1}{t_1}\binom{w_2-1}{t_2}.
\end{align}

By combining the above calculations, we obtain the theorem.
\end{proof}

As a special case of Theorem~\ref{thm:ribbon_2_corner},
we obtain the following sum formula of bounded type
for the hook shape. Note that the left hand side of the next corollary corresponds to the hook shape $(s,1^{r})$ when $s\ge2$.

\begin{cor}
\label{cor:hook_sum_formula}
For $s,r\ge1$ and $w\ge s+r+2$, we have 
\begin{align}
S_w
\begin{varray}
 0, & s-1\\
r, & 1
\end{varray}
&=\binom{w-2}{s-1}\zeta(w)
-\sum^{s-1}_{k=1}\binom{w-k-2}{s-k-1}\zeta(k,w-k)
+(-1)^{s}\sum^{s+r-1}_{k=s}\zeta(k,w-k)\\
&\ \ \ -\sum^{s-1}_{k=2}(-1)^k\binom{w-k-2}{s-k-1}\zeta(k)\zeta(w-k)
-\sum^{r}_{k=2}\binom{w-k-2}{s-2}\zeta(k)\zeta(w-k)\\
&\ \ \ +(-1)^{r}\sum^{s+r-1}_{k=r+1}(-1)^k\binom{w-k-2}{r+s-1-k}\zeta(k)\zeta(w-k).
\end{align}
\end{cor}
\begin{proof}
Follows from \cref{thm:ribbon_2_corner}
and some rearrangement of terms.
\end{proof}

\begin{example}
\label{ex:hook case}
For the shape $\lambda = (3,1,1)$ we obtain for $w\ge 7$ the sum formula
\begin{align}
S_{w}\left(\,
{\footnotesize
\ytableausetup{centertableaux, boxsize=0.6em}
\begin{ytableau}
\, & \,  &\, \\
\, \\
\, \\
\end{ytableau}}\,
\right) &= S_w
\begin{varray}
 0, & 2\\
2, & 1
\end{varray}\\
&=\binom{w-2}{2}\zeta(w)-(w-3)\zeta(2)\zeta(w-2)-(w-5)\zeta(3)\zeta(w-3)+\zeta(4)\zeta(w-4)
\\
&\hspace{0.2cm}-(w-3)\zeta(1,w-1)-\zeta(2,w-2)-\zeta(3,w-3)-\zeta(4,w-4)
\end{align}
from \cref{cor:hook_sum_formula} by taking $s=3$ and $r=2$.
\end{example}

\section{Diagrams with one corner}
\newcommand{\HH}{\mathfrak{H}}
\label{sec:OneCorner}
In this section, we consider general shapes with one corner. Recall that in Section \ref{sec:weightedsum} we defined for an index $\bk=(k_1,\dots,k_d)$ and $l\geq 0$
\begin{align}\label{eq:defQl}
Q_l(\bk)&=
\sum_{\substack{\bw=(w_1,\ldots,w_d):\text{ adm.}\\ \wt(\bw)=\wt(\bk)+l}}
\binom{w_1-1}{k_1-1}\cdots \binom{w_{d-1}-1}{k_{d-1}-1}\binom{w_{d}-2}{k_{d}-1}
\zeta(\bw)\,.
\end{align}
We will show in Theorem \ref{thm:onecornerwithphi} that our target $S_{w}(\lambda/\mu)$ defined in \eqref{eq:sk} with $\lambda/\mu$ having one corner can be expressed explicitly in terms of $Q_{w-|\lambda/\mu|}$. This will follow from a purely combinatorial argument and we will see that this statement is already true for truncated Schur MZVs. For an integer $M\geq 1$ and an arbitrary, not necessarily admissible, Young tableau $\bk$, these are defined by 
\begin{align}\label{eq:deftruncschurmzv}
    \zeta_M(\bk) = \sum_{\substack{(m_{i,j}) \in \SSYT(\lambda \slash \mu)\\m_{i,j}<M}} \prod_{(i,j) \in D(\lambda \slash \mu)}  \frac{1}{m_{i,j}^{k_{i,j}}} \,.
\end{align}
In particular, for an integer $M\geq 1$ and an index $\bk = (k_1,\dots,k_d)$ with $k_1,\dots,k_d\geq 1$ the truncated MZVs are given by 
 \begin{align}
     \zeta_M(k_1,\dots,k_d) = \sum_{0 < m_1 < \dots < m_d < M} \frac{1}{m_1^{k_1}\dots m_d^{k_d}}\,.
 \end{align}
To give the precise statement of above mentioned theorem, we will need to introduce some algebraic setup following \cite{Hoffman97}. 
Denote by $\HH^1=\Q\langle z_k \mid k\geq 1\rangle$ the non-commutative polynomial ring in the variables $z_k$ for $k\geq1$. A monic monomial in $\HH^1$ is called a word and the empty word will be denoted by ${\bf 1}$. 
For $k\geq 1$ and $n\in \Z$, we define
\begin{align}
    z_k^n = \left\{\begin{array}{cl}
        \underbrace{z_k\cdots z_k}_{n}&\text{if $n>0$},\\[7mm]
        \mathbf{1}&\text{if $n=0$},\\[2mm]
        0&\text{if $n<0$}.
    \end{array}\right.
\end{align}

There is a one-to-one correspondence between indices and words; 
to each index $\bk=(k_1,\ldots,k_d)$ corresponds the word 
$z_\bk\coloneqq z_{k_1}\cdots z_{k_d}$. 
Thus we can extend various functions on indices to $\Q$-linear maps on $\HH^1$. 
For example, we define $Q_l\colon\HH^1\to\R$ by setting $Q_l(z_\bk)=Q_l(\bk)$ and 
extending it linearly. 

We define the stuffle product $\ast$ and the index shuffle product $\tsh$ on $\HH^1$ as the $\Q$-bilinear products, which satisfy ${\bf 1} \ast w = w \ast {\bf 1} = w$ and ${\bf 1} \tsh w = w \tsh {\bf 1} = w$ for any word $w\in \HH^1$ and for any $i,j \geq 1$ and words $w_1,w_2 \in \HH^1$,
\begin{align}
z_{i} w_1 \ast z_{j} w_2 &= z_i (w_1 \ast z_{j} w_2) + z_j (z_i w_1 \ast w_2) + z_{i+j} (w_1 \ast  w_2)\,, \\
z_{i} w_1 \tsh z_{j} w_2 &= z_i (w_1 \tsh z_{j} w_2) + z_j (z_i w_1 \tsh w_2)\,.
\end{align}
By \cite[Theorem 2.1]{Hoffman97} we obtain a commutative $\Q$-algebra $\HH^{1}_{\ast}$.

\begin{lemma}\label{lem:SSD}
Let $D_1,\ldots,D_r$ be non-empty subsets of the skew Young diagram $D(\lambda/\mu)$ which gives a disjoint decomposition 
of $D(\lambda/\mu)$, i.e., $D(\lambda/\mu)=D_1\sqcup\dots\sqcup D_r$. 
Then the following conditions are equivalent: 
\begin{enumerate}[label=\textup{(\roman*)}]
\item If we set $t_{ij}=a$ for $(i,j)\in D_a$ with $a=1,\ldots,r$, 
then $(t_{ij})$ is a semi-standard Young tableau of shape $\lambda/\mu$. 
\item There exists a semi-standard Young tableau $(m_{ij})$ of shape $\lambda/\mu$ such that 
\[m_{ij}<m_{kl}\iff a<b\]
holds for any $(i,j)\in D_a$ and $(k,l)\in D_b$. 
\end{enumerate}
\end{lemma}

\begin{proof}
Obviously, (i) implies (ii). Conversely, assume that a semi-standard tableau $(m_{ij})$ satisfies 
the condition in (ii). If $(t_{ij})$ is defined as in (i), one has 
\[m_{ij}<m_{kl}\iff a<b\iff t_{ij}<t_{kl}\]
for any $(i,j)\in D_a$ and $(k,l)\in D_b$, which shows that $(t_{ij})$ is also semi-standard. 
Thus (ii) implies (i). 
\end{proof}

We call a tuple $(D_1,\ldots,D_r)$ of non-empty subsets of $D(\lambda/\mu)$ 
satisfying the equivalent conditions (i) and (ii) of Lemma \ref{lem:SSD} a \emph{semi-standard decomposition}. 
Let $\SSD(\lambda/\mu)$ denote the set of all semi-standard decompositions of $D(\lambda/\mu)$. 
Then we define an element $\varphi(\lambda/\mu)$ of $\mathfrak{H}^1$ by 
\[\varphi(\lambda/\mu)\coloneqq \sum_{(D_1,\ldots,D_r)\in\SSD(\lambda/\mu)}
z_{\abs{D_1}}\cdots z_{\abs{D_r}}, \]
where $\abs{D_i}$ denotes the number of elements of $D_i$. 
This element is related to the sum formula as follows. 

\begin{thm}\label{thm:onecornerwithphi}
When $\lambda/\mu$ has only one corner, we have for $w>\abs{\lambda/\mu}$
\[S_w(\lambda/\mu)=Q_{w-\abs{\lambda/\mu}}(\varphi(\lambda/\mu)). \]
\end{thm}

\begin{proof}
For any admissible Young tableau $\bk=(k_{ij})$ of shape $\lambda/\mu$, we see that  
\begin{equation}\label{eq:zeta(bk) via SSD}
\zeta(\bk)=\sum_{(D_1,\ldots,D_r)\in\SSD(\lambda/\mu)}
\zeta\Biggl(\sum_{(i,j)\in D_1}k_{ij},\ldots,\sum_{(i,j)\in D_r}k_{ij}\Biggr) 
\end{equation}
by classifying the semi-standard tableaux $(m_{ij})$ of shape $\lambda/\mu$ 
according to the semi-standard decompositions $(D_1,\ldots,D_r)$ determined as in Lemma \ref{lem:SSD} (ii). 
Then, from the definition of $S_w(\lambda/\mu)$, we have 
\begin{align}
&S_w(\lambda/\mu)
=\sum_{\substack{\bk \in \YT(\lambda \slash \mu) \\ \bk: \text{ adm.}\\\wt(\bk)=w}} \zeta(\bk)
=\sum_{\substack{k_{ij}\ge 1,\,(i,j)\in D(\lambda/\mu)\\ k_{ij}\ge 2,\,(i,j)\in C(\lambda/\mu)\\ 
\sum_{(i,j)}k_{ij}=w}}
\sum_{(D_1,\ldots,D_r)\in\SSD(\lambda/\mu)}
\zeta\Biggl(\sum_{(i,j)\in D_1}k_{ij},\ldots,\sum_{(i,j)\in D_r}k_{ij}\Biggr)\\
&=\sum_{(D_1,\ldots,D_r)\in\SSD(\lambda/\mu)}
\sum_{\substack{w_1,\ldots,w_{r-1}\ge 1\\ w_r\ge 2\\ w_1+\cdots+w_r=w}}
\binom{w_1-1}{\abs{D_1}-1}\cdots\binom{w_{r-1}-1}{\abs{D_{r-1}}-1}\binom{w_r-2}{\abs{D_r}-1}
\zeta(w_1,\ldots,w_r). 
\end{align}
Here recall that $C(\lambda/\mu)$ is the set of corners of $\lambda/\mu$ and note that, for any $(D_1,\ldots,D_r)\in\SSD(\lambda/\mu)$, the unique corner of $\lambda/\mu$ belongs to $D_r$. 
The last expression above is equal to 
\[\sum_{(D_1,\ldots,D_r)\in\SSD(\lambda/\mu)}
Q_{w-\abs{\lambda/\mu}}(\abs{D_1},\ldots,\abs{D_r}) 
=Q_{w-\abs{\lambda/\mu}}(\varphi(\lambda/\mu)).\]
Thus the proof is complete. 
\end{proof}

By Remark \ref{rem:plisbounded} we see that Theorem \ref{thm:onecornerwithphi} gives a sum formula of bounded type for shapes with one corner. In order to evaluate the sum $S_w(\lambda/\mu)$ in the one corner case, 
one therefore needs to find an expression of $\varphi(\lambda/\mu)$. 
For this purpose, the following expression is useful: 

\begin{prop}\label{prop:jacobitrudiphi}
For any skew shape $\lambda/\mu$, let $\lambda'=(\lambda'_1,\ldots,\lambda'_s)$ 
and $\mu'=(\mu'_1,\ldots,\mu'_s)$ be the conjugates of $\lambda$ and $\mu$, respectively. 
Then we have the identity 
\[\varphi(\lambda/\mu)=\det\nolimits_*\bigl[z_1^{\lambda'_i-\mu'_j-i+j}\bigr]_{1\le i,j\le s}, \]
where $\det\nolimits_*$ denotes the determinant performed in the stuffle algebra $\HH^1_*$.
\end{prop}

\begin{proof}
For any integer $M>0$, we have 
\[\zeta_M(\varphi(\lambda/\mu))
=\sum_{(D_1,\ldots,D_r)\in\SSD(\lambda/\mu)}\zeta_M(\abs{D_1},\ldots,\abs{D_r})
=\zeta_M(\mathbf{O}_{\lambda/\mu}), \]
where $\mathbf{O}_{\lambda/\mu}$ denotes the tableau of shape $\lambda/\mu$ all entries of which are $1$. 
Indeed, the first equality is obvious from the definition of $\varphi(\lambda/\mu)$, 
and the second follows in the same way as \eqref{eq:zeta(bk) via SSD}. 

On the other hand, by the Jacobi-Trudi type formula for truncated Schur MZVs 
(\cite[Theorem 1.1]{NakasujiPhuksuwanYamasaki2018}, \cite[Theorem 4.7]{Bachmann2018}), we have 
\[\zeta_M(\mathbf{O}_{\lambda/\mu})
=\det\bigl[\zeta_M(z_1^{\lambda'_i-\mu'_j-i+j})\bigr]_{1\le i,j\le s}
=\zeta_M\Bigl(\det\nolimits_*\bigl[z_1^{\lambda'_i-\mu'_j-i+j}\bigr]_{1\le i,j\le s}\Bigr).\]
Since the map $\HH^1 \rightarrow \prod_{M\geq 1} \Q$ given by $w \mapsto (\zeta_M(w))_{M\geq 1}$ is 
injective (\cite[Theorem 3.1]{Yamamoto13}), we obtain the desired identity in $\HH^1$. 
\end{proof}

\ytableausetup{boxsize=5pt}
In some cases, one can compute $\varphi(\lambda/\mu)$ explicitly 
by using Proposition \ref{prop:jacobitrudiphi}. 

\begin{thm}
For $n\ge 1$ and $k\ge 0$, it holds that 
\begin{align}
\varphi
\bigl(\,(2^{n+k})/(1^k)\,\bigr)
&=
\sum^{n}_{l=0}
\frac{k+1}{l+k+1}\binom{2l+k}{l} z_2^{n-l}\tsh z_1^{2l+k}.
\end{align}
\end{thm}
\begin{proof}
By \cref{prop:jacobitrudiphi}, we have 
\begin{align}
\varphi\bigl(\,(2^{n+k})/(1^k)\,\bigr)
=
\begin{vmatrix}
 z^n_1 & z^{n+k+1}_1 \\
 z^{n-1}_1  & z^{n+k}_1
\end{vmatrix}
=z^{n+k}_1 \ast z^{n}_1 - z^{n+k+1}_1 \ast z^{n-1}_1.
\end{align} 
By \cite[Lemma 1]{Chen15} we get for $m\ge n\ge 1$
\begin{align}
\label{for:z1z1}
 z^m_1\ast z^n_1
&=\sum^{n}_{l=0}\binom{m+n-2l}{m-l}\sum_{\bk\in G^{m+n-2l}_{l}}z_\bk
=\sum^{n}_{l=0}\binom{m+n-2l}{n-l}z_2^{l} \tsh z_1^{m+n-2l},
\end{align}
 where $G^a_b$ is the set of all possible indices containing $a$ times $1$ and $b$ times $2$. This gives 
\begin{align}
&\varphi\bigl(\,(2^{n+k})/(1^k)\,\bigr)\\
&=z^{n+k}_1 \ast z^{n}_1 - z^{n+k+1}_1 \ast z^{n-1}_1\\
&=
\sum^{n}_{l=0}\binom{2n+k-2l}{n-l}
z_2^{l}\tsh z_1^{2n+k-2l}
-
\sum^{n-1}_{l=0}\binom{2n+k-2l}{n-l-1}
z_2^{l}\tsh z_1^{2n+k-2l}\\
&=\sum^{n}_{l=0}
\left(\binom{2n+k-2l}{n-l}-\binom{2n+k-2l}{n-l-1}\right)
z_2^{l}\tsh z_1^{2n+k-2l}\\
&=\sum^{n}_{l=0}
\left(\binom{2l+k}{l}-\binom{2l+k}{l-1}\right)
z_2^{n-l}\tsh z_1^{2l+k}\\
&=\sum^{n}_{l=0}
\frac{k+1}{l+k+1}\binom{2l+k}{l} z_2^{n-l}\tsh z_1^{2l+k}.\qedhere
\end{align} 
\end{proof}

For some shapes $\lambda/\mu$, the $\varphi(\lambda/\mu)$ will contain sums over all indices over a fixed weight and depth. The corresponding sums of $Q$ applied to these can be evaluated by using the following lemma.
\begin{lemma}
\label{lem:binomprodsum}
 For $k\ge d\ge1, w\ge d+1$ and an index $\bn=(n_1,\dots,n_d)$, we have
\begin{align}
\sum_{\substack{\bk=(k_1,\dots,k_d)\\\wt(\bk)=k}}P_{w-k}(\bn;\bk)
=\binom{w-\wt(\bn)}{k-d}\zeta(w).
\end{align}
 In particular, 
\[
 \sum_{\substack{\bk=(k_1,\dots,k_d)\\\wt(\bk)=k}}P_{w-k}(\bk)
=\binom{w-d}{k-d}\zeta(w), \qquad
 \sum_{\substack{\bk=(k_1,\dots,k_d)\\\wt(\bk)=k}}Q_{w-k}(\bk)
=\binom{w-d-1}{k-d}\zeta(w).
\]
\end{lemma}
\begin{proof}
This follows directly by using the Chu-Vandermonde identity and the usual sum formula for MZVs. 
\end{proof}

\begin{remark}
For $s\ge 0$ and $r\ge 1$, one can show the identity 
\[
\varphi\big(((s+1)^{r})/(s^{r-1})\big)
=\sum^{s}_{l=0}\binom{r+l-1}{l}
\sum_{\substack{\bk=(k_1,\ldots,k_{r+l})\\ \wt(\bk)=r+s}}z_\bk
\]
by counting how many times $z_\bk$ appears for each $\bk$, as in the proof of Theorem \ref{thm:s00}.
(One can also show this by induction on $s$ via the expression
\begin{align}
\varphi\big(((s+1)^{r})/(s^{r-1})\big)
&=\det\nolimits_*\!
\left[
\begin{array}{ccccc}
z_1 & z^2_1 & \cdots & z^s_1 & z^{r+s}_{1}  \\
1 & z_1 & \cdots & z^{s-1}_1 & z^{r+s-1}_1 \\
0 & \ddots & \ddots & \vdots & \vdots \\
\vdots & \ddots & \ddots & z_1 & z^{r+1}_1 \\
0 & \cdots & 0 & 1 & z^r_1 
\end{array}
\right]
\end{align}
obtained from Proposition~\ref{prop:jacobitrudiphi}
together with \eqref{for:z1z1}.)
Using Theorem \ref{thm:onecornerwithphi} and  Lemma~\ref{lem:binomprodsum} this gives another way of proving the anti-hook sum formula in Theorem \ref{thm:anti-hook}
\begin{align}
 S_{w}\big(((s+1)^{r})/(s^{r-1})\big)&=Q_{w-(r+s)}\left(\varphi\big(((s+1)^{r})/(s^{r-1})\big)\right)\\
&=\sum^{s}_{l=0}\binom{r+l-1}{l}
\sum_{\substack{\bk=(k_1,\ldots,k_{r+l})\\ \wt(\bk)=r+s}}Q_{w-(r+s)}(\bk)\\
&=\sum^{s}_{l=0}\binom{r+l-1}{l}
\binom{w-r-l-1}{s-l}\zeta(w)\\
&=\binom{w-1}{s}\zeta(w)\,.
\end{align}
\end{remark}

We can summarize the general strategy to give a bounded expression of $S_{w}(\lambda/\mu)$ for the case when $\lambda/\mu$ has one corner as follows:
\begin{enumerate}[label=\textup{(\roman*)}]
\item Give an expression for $\varphi(\lambda/\mu)$, by evaluating the determinant in Proposition \ref{prop:jacobitrudiphi} by the stuffle product. Then use Theorem \ref{thm:onecornerwithphi} to get $S_{w}(\lambda/\mu)=Q_{w-|\lambda/\mu|}(\varphi(\lambda/\mu))$.
\item If sums of $Q_l$ over all indices of a fixed weight and depth appear, use Lemma \ref{lem:binomprodsum} to write them in terms of Riemann zeta values.
\item For other terms involving $Q$ write them in terms of $P$, by using \eqref{for:PtoQ}, i.e.,
\begin{equation}
 Q_l(k_1,\dots,k_d)
 =\sum^{k_d-1}_{j=0}(-1)^jP_{l+j}(k_1,\ldots,k_{d-1},k_d-j).
\end{equation}
Then use Proposition \ref{prop:bwsumformula} to get recursive (bounded) expressions of $P_{l+j}$ in terms of MZVs. For depth $2$ and $3$ explicit expressions are given by Corollary \ref{cor:pldep2} and \ref{cor:pldep3}.
\end{enumerate}

\begin{example} Using above strategy we get formula \eqref{eq:22square} for $  S_w\left(\,{\footnotesize \ytableausetup{centertableaux, boxsize=0.8em}
	\begin{ytableau}
	\, & 	\,  \\
	\, & 	\,
	\end{ytableau}}\,\right)$ in the introduction and the following examples:
\begin{enumerate}[label=\textup{(\roman*)}]
\item For $w\ge 6$ we have
\begin{align}
S_w\left(\
	\begin{ytableau}
\none &	\, & 	\,  \\
\,&	\, & 	\,
	\end{ytableau}\ \right) 
&=\binom{w-2}{2}\zeta(2) \zeta(w-2)-\frac{5}{4}\zeta(4)\zeta(w-4)+\zeta(2) \zeta(2,w-4)
\\
&\ \ \ -\zeta(2) \zeta(1,w-3)- \binom{w-2}{2}\zeta(1,w-1)
+\binom{w-3}{2} \zeta(2,w-2)\\
&\ \ \ +(w-3) \zeta(3,w-3)+(w-3) \zeta(1,1,w-2)
-(w-5) \zeta(1,2,w-3)\\
&\ \ \ -2 \zeta(1,3,w-4)+\zeta(2,1,w-3)-\zeta(2,2,w-4).\\
\end{align}
\item For $w\ge 6$ we have
\begin{align}
S_w&\left(\ 
	\begin{ytableau}
	\none & \,  \\
	\, & 	\,  \\
	\, & 	\,
	\end{ytableau}\ \right) 
=(w-2) \zeta(2) \zeta(w-2)+(w-5)\zeta(3)\zeta(w-3)-\frac{5}{4}\zeta(4)\zeta(w-4)\\
&\ \ \ -\zeta(2) \zeta(1,w-3)+\zeta(2) \zeta(2,w-4)+(2-w)\zeta(1,w-1)+(w-4) \zeta(2,w-2)\\
&\ \ \ +2 \zeta(3,w-3)+(w-3) \zeta(1,1,w-2)-(w-5) \zeta(1,2,w-3)\\
&\ \ \ -2 \zeta(1,3,w-4)+\zeta(2,1,w-3)-\zeta(2,2,w-4).
\end{align}
\end{enumerate}
Comparing (i) and (ii) with \eqref{eq:22square}, we see that for all $w \geq 1$ we have
\begin{align}\label{eq:examplerelationofs}
 2 S_w\left(\ 
	\begin{ytableau}
\none &	\, & 	\,  \\
\,&	\, & 	\,
	\end{ytableau}\ \right) -   2S_w\left(\ 
	\begin{ytableau}
	\none & \,  \\
	\, & 	\,  \\
	\, & 	\,
	\end{ytableau}\ \right)  
 &= (w-5) S_w\left(\,{\footnotesize 
	\begin{ytableau}
	\, & 	\,  \\
	\, & 	\,
	\end{ytableau}}\,\right)\,.
\end{align}
\end{example}

We will now show that the relation \eqref{eq:examplerelationofs} among $S_w$ for different shapes is a special case of a more general family of relations. For this we first notice that any skew Young diagram $D(\lambda/\mu)$ with one corner can be written as 
\[\lambda=(n^m)=(\underbrace{n,\ldots,n}_m),\quad \mu=(\mu_1,\ldots,\mu_m)\]
with $\mu_1=n$, $\mu_m>0$ (for example, the one-box diagram is represented as $(2,2)/(2,1)$). 
Then we write $I$ for the set of $i=1,\ldots,m$ that 
\[\mu[i]\coloneqq(\mu_1,\ldots,\mu_i-1,\ldots,\mu_m)\]
is non-increasing. 
Then, for $i\in I$,  $D(\lambda/\mu[i])=D(\lambda/\mu)\cup\{(i,\mu_i)\}$ 
is still a skew Young diagram with one corner.  As a generalization of \eqref{eq:examplerelationofs} we obtain the following.

\begin{thm}\label{thm:S_w rel} If $\lambda/\mu$ is a skew Young diagram with one corner then for all $w\geq 1$ 
\begin{equation}\label{eq:S_w rel}
\sum_{i\in I}\bigl((i-\mu_i)-(m-n)\bigr)S_w(\lambda/\mu[i])
=(w-\lvert\lambda/\mu\rvert-1)S_w(\lambda/\mu). 
\end{equation}
\end{thm}

To prove this, we need the following Lemma. 
Define a linear map $\partial\colon\HH^1\to\HH^1$ by 
\[\partial(z_{k_1}\cdots z_{k_d})\coloneqq 
\sum_{a=1}^d k_a z_{k_1}\cdots z_{k_a+1}\cdots z_{k_d}.\]
In particular, $\partial(1)=0$. 

\begin{lemma}\label{lem:partial}\leavevmode
\begin{enumerate}[label=\textup{(\roman*)}]
\item\label{lem:partial:derivation}
$\partial$ is a derivation with respect to the stuffle product. 
\item\label{lem:partial:power}
For any $N\in\Z$, we have 
\[\partial(z_1^N)=z_1*z_1^N-(N+1)z_1^{N+1}.\]
\item We have 
\[(l-1)Q_l(v)=Q_{l-1}(\partial(v))\]
for any $v\in\HH^1$. 
\end{enumerate}
\end{lemma}
\begin{proof}
\cref{lem:partial:derivation} is verified, e.g., by induction on the depth. 
It is also easy to show (ii) from the definition. 
Finally, the identity in (iii) follows from the definition of $Q_l$ and the identity 
\begin{align*}
&(w-k-1)\binom{w_1-1}{k_1-1}\cdots\binom{w_d-2}{k_d-1}\\
&=\bigl((w_1-k_1)+\cdots+(w_d-1-k_d)\bigr)
\binom{w_1-1}{k_1-1}\cdots\binom{w_d-2}{k_d-1}\\
&=\sum_{a=1}^d
k_a\binom{w_1-1}{k_1-1}\cdots\binom{w_a-1}{k_a}\cdots\binom{w_d-2}{k_d-1}, 
\end{align*}
where $w=w_1+\cdots+w_d$ and $k=k_1+\cdots+k_d$. 
\end{proof}

\begin{proof}[Proof of Theorem \ref{thm:S_w rel}]
By \cref{thm:onecornerwithphi} and \cref{prop:jacobitrudiphi}, we have 
\[S_w(\lambda/\mu)=Q_{w-\abs{\lambda/\mu}}
\Bigl(\det\nolimits_*\bigl[z_1^{m-\mu'_j-i+j}\bigr]_{1\le i,j\le n}\Bigr). \]
Here $(\mu'_1,\ldots,\mu'_n)$ denotes the transpose of the partition $\mu=(\mu_1,\ldots,\mu_m)$. 
Note that $\mu'_1=m$ and $\mu'_n>0$. 
By \cref{lem:partial} (iii) and (i), we see that  
\begin{align}\label{eq:S_w rel RHS1}
(w-\lvert\lambda/\mu\rvert-1)S_w(\lambda/\mu)
&=(w-\lvert\lambda/\mu\rvert-1)Q_{w-\abs{\lambda/\mu}}
\Bigl(\det\nolimits_*\bigl[z_1^{m-\mu'_j-i+j}\bigr]_{1\le i,j\le n}\Bigr)\\
&=Q_{w-\abs{\lambda/\mu}-1}
\Bigl(\partial\det\nolimits_*\bigl[z_1^{m-\mu'_j-i+j}\bigr]_{1\le i,j\le n}\Bigr)\\
&=\sum_{k=1}^n Q_{w-\abs{\lambda/\mu}-1}
\Bigl(\det\nolimits_*\bigl[\partial^{\delta_{jk}}(z_1^{m-\mu'_j-i+j})\bigr]_{1\le i,j\le n}\Bigr). 
\end{align}
Here $\partial^{\delta_{jk}}$ means the operator $\partial$ if $j=k$ and the identity operator 
otherwise. Moreover, since the identity 
\begin{align*}
\partial (z_1^{m-\mu'_k-i+k})
&=z_1*z_1^{m-\mu'_k-i+k}-(m-\mu'_k-i+k+1)z_1^{m-\mu'_k-i+k+1}\\ 
&=\bigl((\mu'_k-k)-(m-n)\bigr)z_1^{m-(\mu'_k-1)-i+k}\\
&\qquad +z_1*z_1^{m-\mu'_k-i+k}-(n+1-i)z_1^{m-(\mu'_k-1)-i+k}
\end{align*}
holds by \cref{lem:partial} (ii), we have 
\begin{align}\label{eq:S_w rel RHS2}
&\det\nolimits_*\bigl[\partial^{\delta_{jk}}(z_1^{m-\mu'_j-i+j})\bigr]\\
&=\bigl((\mu'_k-k)-(m-n)\bigr)\det\nolimits_*\bigl[z_1^{m-(\mu'_j-\delta_{jk})-i+j}\bigr]\\
&\qquad+z_1*\det\nolimits_*\bigl[z_1^{m-\mu'_j-i+j}\bigr]
-\det\nolimits_*\bigl[(n+1-i)^{\delta_{jk}}z_1^{m-(\mu'_j-\delta_{jk})-i+j}\bigr].
\end{align}

On the other hand, the left-hand side of \eqref{eq:S_w rel} is equal to 
\begin{equation}\label{eq:S_w rel LHS}
\sum_{k=1}^n\bigl((\mu'_k-k)-(m-n)\bigr)
Q_{w-\abs{\lambda/\mu}-1}
\Bigl(\det\nolimits_*\bigl[z_1^{m-(\mu'_j-\delta_{jk})-i+j}\bigr]\Bigr). 
\end{equation}
Here, a priori, $k$ runs only over the indices such that $\mu'_k>\mu'_{k+1}$. 
However, if $\mu'_k=\mu'_{k+1}$, the $k$-th and $(k+1)$-st columns of 
the matrix $\bigl(z_1^{m-(\mu'_j-\delta_{jk})-i+j}\bigr)_{i,j}$ are equal, 
so the determinant is zero. 

Comparing \eqref{eq:S_w rel RHS1}, \eqref{eq:S_w rel RHS2} and \eqref{eq:S_w rel LHS}, 
it suffices to prove the equality 
\begin{equation}\label{eq:S_w rel reduced}
nz_1*\det\nolimits_*\bigl[z_1^{m-\mu'_j-i+j}\bigr]
\overset{?}{=}
\sum_{k=1}^n\det\nolimits_*\bigl[(n+1-i)^{\delta_{jk}}z_1^{m-(\mu'_j-\delta_{jk})-i+j}\bigr]. 
\end{equation}
Let us compute the right-hand side by the cofactor expansion with respect to the $k$-th column. 
\begin{align*}
&\sum_{k=1}^n\det\nolimits_*\bigl[(n+1-i)^{\delta_{jk}}z_1^{m-(\mu'_j-\delta_{jk})-i+j}\bigr]\\
&=\sum_{k=1}^n\sum_{l=1}^n(-1)^{k+l}(n+1-l)
z_1^{m-(\mu'_k-1)-l+k}*\det\nolimits_*\bigl[z_1^{m-\mu'_j-i+j}\bigr]_{i\ne l,j\ne k}\\
&=\sum_{l=1}^n(n+1-l)\sum_{k=1}^n(-1)^{k+l}
z_1^{m-(\mu'_k-1)-l+k}*\det\nolimits_*\bigl[z_1^{m-\mu'_j-i+j}\bigr]_{i\ne l,j\ne k}. 
\end{align*}
Then, by the cofactor expansion with respect to the $l$-th row, we have 
\[\sum_{k=1}^n(-1)^{k+l}
z_1^{m-(\mu'_k-1)-l+k}*\det\nolimits_*\bigl[z_1^{m-\mu'_j-i+j}\bigr]_{i\ne l,j\ne k}
=\det\nolimits_*\bigl[z_1^{m-\mu'_j-(i-\delta_{il})+j}\bigr].\]
This is zero for $l=2,\ldots,n$ since the $l$-th and $(l-1)$-st rows are equal. 
Thus we have shown that the right-hand side of \eqref{eq:S_w rel reduced} is 
$n\det\nolimits_*\bigl[z_1^{m+\delta_{i1}-\mu'_j-i+j}\bigr]$, 
and now it is enough to prove 
\begin{equation}\label{eq:S_w rel last}
z_1*\det\nolimits_*\bigl[z_1^{m-\mu'_j-i+j}\bigr]
\overset{?}{=}
\det\nolimits_*\bigl[z_1^{m+\delta_{i1}-\mu'_j-i+j}\bigr]. 
\end{equation}
But this is obvious since two matrices here are of the form 
\[\begin{pmatrix}
1 & * \\
\textbf{0} & Z
\end{pmatrix}
\quad \text{and} \quad 
\begin{pmatrix}
z_1 & * \\
\textbf{0} & Z
\end{pmatrix},\]
respectively, with the common $(n-1)\times(n-1)$ matrix 
$Z=\bigl[z_1^{m-\mu'_i-i+j}\bigr]_{2\le i,j\le n}$. 
Hence the proof is complete. 
\end{proof}

\end{document}